\documentclass[11pt,a4paper]{article}
\usepackage[margin=30truemm]{geometry}
\usepackage{amsmath,amsthm,amssymb,amsfonts}
\usepackage{enumitem}
\usepackage{color} 
\usepackage{dsfont} 
\usepackage{booktabs,multirow}
\usepackage{url}
\usepackage{authblk}

\usepackage{titlesec}
\titleformat*{\section}{\large\bfseries}
\titleformat*{\subsection}{\normalsize\bfseries}
\titleformat*{\subsubsection}{\normalsize\itshape}

\usepackage{mleftright}\mleftright 

\usepackage{graphicx}
\graphicspath{{./fig/}}

\definecolor{green}{rgb}{0,0.75,0}
\definecolor{purple}{rgb}{0.5,0,0.5}

\newcommand{\1}{\mathds{1}}
\newcommand\dd{\mathrm{d}}
\newcommand{\E}{\mathbb{E}}
\newcommand{\R}{\mathbb{R}}
\renewcommand{\Pr}{\mathbb{P}}
\newcommand{\N}{\mathbb{N}}
\renewcommand{\S}{\mathbb{S}}
\newcommand{\tr}{\mathrm{tr}}

\newcommand{\eqd}{\stackrel{d}{=}}

\newcommand{\Corr}{\mathrm{Corr}}
\newcommand{\Var}{\mathrm{Var}}

\newtheorem{example}{Example}
\newtheorem{theorem}{Theorem}
\newtheorem{lemma}{Lemma}
\newtheorem{proposition}{Proposition}
\newtheorem{remark}{Remark}

\newtheorem{definition}{Definition}
\newtheorem{assumption}{Assumption}

\begin{document}

\title{Tube formula for spherically contoured random fields with subexponential marginals}
\author{Satoshi Kuriki\thanks{The Institute of Statistical Mathematics, Tokyo, Japan, \texttt{kuriki@ism.ac.jp}} 
\thanks{The Graduate University for Advanced Studies, SOKENDAI, Hayama, Japan}, Evgeny Spodarev\thanks{Institute of Stochastics, Ulm University, Germany, \texttt{evgeny.spodarev@uni-ulm.de}}}

\date{}

\maketitle

\begin{abstract}
It is widely known that the tube method, or equivalently the Euler characteristic heuristic, provides a very accurate approximation for the tail probability that the supremum of a smooth Gaussian random field exceeds a threshold value $c$.
The relative approximation error $\Delta(c)$
is exponentially small as a function of $c$ when $c$ tends to infinity.
On the other hand, little is known about non-Gaussian random fields.

In this paper, we obtain the approximation error of the tube method applied to 
the canonical isotropic random fields on a unit sphere
defined by $u\mapsto\langle u,\xi\rangle$, $u\in M\subset\mathbb{S}^{n-1}$, where $\xi$ is a spherically contoured random vector.
These random fields have statistical applications in multiple testing and simultaneous regression inference when the unknown variance is estimated.
The decay rate of the relative error $\Delta(c)$ depends on the tail of the distribution of $\|\xi\|^2$ and the critical radius of the index set $M$.
If this distribution is subexponential but not regularly varying, $\Delta(c)\to 0$ as $c\to\infty$. 
However, in the regularly varying case, $\Delta(c)$ does not vanish and hence is not negligible.
To address this limitation, we provide simple upper and lower bounds for $\Delta(c)$
and for the tube formula itself.
Numerical studies are conducted to assess the accuracy of the asymptotic approximation.

\end{abstract}


\medskip


\section{Introduction}

Let $\eta$ be a random vector uniformly distributed on the unit sphere $\S^{n-1}$ in $n$-dimensional Euclidean space $\R^n$.
For a positive random variable $r$ independent of $\eta$, set $\xi=r\eta$.
Define a random field on $\S^{n-1}$ by
\begin{equation}
\label{Xi_u}
 \Xi_u := \langle u,\xi\rangle = r \langle u,\eta\rangle, \quad u\in\S^{n-1}.
\end{equation}
The random field (\ref{Xi_u}) is isotropic, that is, invariant under the orthogonal transformations $u\mapsto g u$, $g\in O(n)$ of indices $u\in\S^{n-1}$.
We call the random field $\Xi$ in (\ref{Xi_u}) a \emph{spherically contoured random field}
because the distribution of the random vector $\xi$ is spherically contoured.
The random field $\Xi$ in (\ref{Xi_u}) is a special case of elliptically contoured random fields (see \cite{jensen-foutz:1989}).
Different from the standard spherical harmonics setting on the sphere, where second-order integrability is assumed (e.g., \cite{marinucci-peccati:2011}), we explicitly allow for heavy-tailed $r=\|\xi\|$ with $\E r^2=+\infty$.
A class of random fields defined by
\begin{equation}
\label{Xi_u-again}
 \Xi_u = \langle u,\zeta\rangle/s, \quad u\in\S^{n-1},
\end{equation}
where $\zeta$ is an $n$-dimensional standard Gaussian vector, and
$s$ is a positive random variable stochastically independent of $\zeta$,
constitutes a proper subclass of the spherical contoured random field (\ref{Xi_u}).
Indeed, when $\Vert\xi\Vert=r$ has a compact support, $\xi$ cannot be expressed as $\zeta/s$ in (\ref{Xi_u-again}).
For other examples of $\xi$ not expressible as $\zeta/s$, see, e.g., \cite[Theorem 2.1]{Gomez-etal:2008}.
In statistical applications, $s$ in (\ref{Xi_u-again}) can be interpreted as an estimator of the unknown scale parameter of $\zeta$ (Section \ref{sec:pre}).

Let $M$ be a closed subset of $\S^{n-1}$
satisfying some smoothness conditions to be described in detail later.
In this paper, we consider the upper tail probability of the supremum on $M$:
\begin{equation}
\label{sup_Xi}
 P(c) := \Pr\biggl(\sup_{u\in M} \Xi_u \ge c\biggr).
\end{equation}

Here we note that the form (\ref{sup_Xi}) is  canonical in the following sense.
Consider a random field defined by a centered Gaussian random field $(f_t)_{t\in T}$ multiplied by an independent random variable $s^{-1}$ (e.g., $s^{-2}$ is totally skewed $\alpha/2$-stable distributed in the $\alpha$-stable sub-Gaussian random field case \cite{adler-samorodnitsky-taylor:2010}, $0<\alpha<2$).
Assume that $f_t$ is unit-variance having a finite Karhunen–Lo\`{e}ve expansion $f_t=\langle \phi(t),\zeta\rangle_{\R^n}$, where $\phi(t)=(\phi_i(t))_{1\le i\le n}$.
Then, the supremum of such a random field over the index set $T$ is written as
\begin{equation*}
 \sup_{t\in T}s^{-1}f_t
 = \sup_{u\in M}\Xi_u, \quad\mbox{where}\ \ %
 M:=\phi(T)\subset \S^{n-1}.
\end{equation*}

The upper level set of the random function $u\mapsto \Xi_u$ with threshold $c$,
\[
 M_c := \{u\in M\mid \Xi_u=\langle u,\xi\rangle \ge c\},
\]
is referred to as the \emph{excursion set} of $\Xi$ on $M$ at a level $c>0$.
Since $P(c)=\Pr(M_c\ne\emptyset)$, it is alternatively called the \emph{excursion probability}.
The random field $\Xi$ does not attain its maximum at two points a.s.
Hence, for large $c$, $M_c$ will be contractible to a point (i.e., the maximizer of $\Xi$) or will be empty.
Therefore, we have
\[
 \chi(M_c) \approx \1(M_c\ne\emptyset) \quad \mbox{when $c$ is large},
\]
where $\chi(\cdot)$ is the Euler-Poincar\'e characteristic of a set.
By taking expectations on both sides,
we have $\E[\chi(M_c)] \approx \Pr(M_c\ne\emptyset)=P(c)$.
This approximation is known as the \textit{Euler characteristic heuristic}.

An alternative approach, independently proposed for approximating the probability $P(c)$ in (\ref{sup_Xi}),
is the \textit{tube method} (e.g., \cite{johnstone-siegmund:1989,johansen-johnstone:1990,sun:1993}).
In \cite{takemura-kuriki:2002}, it was shown that the tube method is equivalent to the Euler characteristic heuristic.
This method applies to random fields generated by finite-dimensional random vectors, such as $\Xi$ in $(\ref{Xi_u})$, and provides a framework for error analysis (e.g., \cite{kuriki-takemura:2001,kuriki-takemura:2008,kuriki-takemura:2009}).  
In this paper, we examine the approximation formula
\begin{equation}
\label{Ptube0}
 P_\mathrm{tube}(c):=\E[\chi(M_c)]\approx P(c)
\end{equation}
within the tube method.
 
It is widely known that the tube method or the Euler characteristic heuristic yield  a very accurate approximation for the upper tail probability (\ref{sup_Xi}) of the supremum of a smooth Gaussian random field such as $\Xi_u$ in (\ref{Xi_u}).
The \emph{relative approximation error}
\[
 \Delta(c) := \frac{P_\mathrm{tube}(c)-P(c)}{P_\mathrm{tube}(c)}
\]
with $c\to +\infty$ is exponentially small as a function of $c$, and its exponential decay rate is determined by the critical radius (reach) of the index set $M$ of the random field \cite{kuriki-takemura:2001,taylor-takemura-adler:2005}.
On the other hand, little is known about non-Gaussian random fields
\cite{adler-samorodnitsky-taylor:2010,adler-samorodnitsky-taylor:2013}.
We call the approximation $P_\mathrm{tube}$ \emph{tail-valid} if
$\Delta(c)\to 0$ as $c\to +\infty$, meaning that $P_\mathrm{tube}(c)$ and $P(c)$ are asymptotically equivalent.
The tail-validity would be a desirable property for the tail probability approximation.
We call the approximation $P_\mathrm{tube}$ \emph{not tail-valid} or \emph{tail-invalid} if
$\Delta(c)\not\rightarrow 0$ as $c\to +\infty$.

In this paper,
we examine the asymptotic behavior of $\Delta(c)$ for large $c$ in the case of several non-Gaussian marginal distributions of $\Xi$.
To this end, we allow for a wide class $\mathcal{L}_{*,*}$ of (heavy as well as light) tails of $R_n=\Vert\xi\Vert^2$
including subexponential (and hence also regularly varying) ones.
For instance, the supremum of a sub-Gaussian random field can be handled within this framework. 
We will show that the tube approximation is tail-valid except for the regularly varying case.
In the case of tail-validity, we observe  a variety of  convergence rates of $\Delta(c)$ to zero.

The paper starts with two motivating statistical applications (multiple testing and simultaneous inference in nonlinear regression) in Section \ref{sec:examples}.
 Section \ref{sec:pre} provides preliminaries on the class of random fields, regularity conditions, tube method and Euler characteristic heuristic.
In Section \ref{sec:heavy-tailed}, we introduce a class of tail probabilities $\mathcal{L}_{*,*}$ for the distribution of $R_n$.
Section \ref{sec:Main} contains the formulation of the main results about the asymptotics of the relative error $\Delta(c)$ as $c\to +\infty$.
We fully examine the case of a finite set $M$,
where the tube formula reduces to the Bonferroni approximation.
In some practical scenarios described in Section \ref{sec:examples},
$s^2$ in (\ref{Xi_u-again}) follows a scaled chi-square distribution, and consequently, $R_n$ follows a regularly varying $F$-distribution.
In such cases, the tube method does not approximate the probability $P(c)$ well.
Here, we propose an upper bound for the error $|\Delta(c)|$
 (Remarks \ref{rem:bound0} and \ref{rem:bound}).
The relationship between the relative error $\Delta(c)$ and the asymptotic (in)dependence in extreme value theory is discussed in Remark \ref{rem:copula}.
The results of numerical studies for the Bonferroni approximation follow in Section \ref{sec:num}.
Proofs of the main results are provided in Section \ref{sec:proof}.
The paper concludes with a short summary in Section \ref{sec:summary}.

\section{Motivating examples}
\label{sec:examples}

The random field (\ref{Xi_u-again}) appears whenever  a Gaussian random field is normalized by an estimator of its unknown variance.
We illustrate this with two examples.

The first example concerns the Bonferroni method for multiple testing.
There, the maximum test statistic
$T_{\max}=\max_{1\le i\le N} T_i$
often occurs, where $T_i$ are test statistics in $i$th test.
It is assumed that
$(T_i)_{1\le i\le N}$ are jointly Gaussian random variables with covariance structure $\Var(T_i)=\sigma^2$ and $\Corr(T_i,T_j)=\rho_{ij}$,
and under the null hypothesis, their means are $\E[T_i]=0$.
In a typical setting of linear models,
$\rho_{ij}$ is determined by the design matrix and hence known.
On the other hand, $\sigma^2$ is an unknown variance of measurement error.
When an estimator $\widehat\sigma$ of $\sigma$ stochastically independent of $T_i$'s is available, we redefine $T_i$ by $T_i/\widehat\sigma$ to remove the nuisance parameter $\sigma^2$.
One possible way to obtain such an estimator $\widehat\sigma$ is by the replication of measurements.
By choosing unit vectors $u_i\in\S^{n-1}$ so that $\langle u_i,u_j\rangle=\rho_{ij}$, it holds that
$(T_i)_{1\le i\le N}\eqd (\Xi_{u_i})_{1\le i\le N}$,
where $\Xi_{u}=\langle u,\zeta\rangle/s$ with $s=\widehat\sigma/\sigma$.
Here and in what follows, ``$\eqd$'' denotes the equality in distribution.
Hence,
\[
 T_{\max} \mathop{=}^d \max_{u\in M}\Xi_u, \quad M=\{u_i \mid 1\le i\le N\}\subset\S^{n-1},
\]
is the maximum of the random field $\Xi$ on $M$.
The tube method for approximating $P(c)=\Pr(T_{\max}\ge c)$ is nothing but the \emph{Bonferroni approximation}
$P_\mathrm{tube}(c) = \sum_{i=1}^N \Pr(T_i\ge c)$.
The Bonferroni approximation is often used because of its conservativeness $P(c)\le P_\mathrm{tube}(c)$ for all $c$.
However, its approximation error including its tail validity $P_\mathrm{tube}(c)\sim P(c)$ is not clear in general.

The second example is about the simultaneous confidence band in nonlinear (curvilinear) regression \cite{naiman:1986,lu-kuriki:2017}.
We consider the case where the measurement error variance $\sigma^2$ is unknown.
Suppose that the pairs $(x_i,y_i)$ of explanatory variables $x_i\in\mathcal{X}\subset\R^n$ and response variables $y_i\in\R$, $i=1,\ldots,N$, are observed.
We consider the regression model
\[
 y_i = b^\top f(x_i) + \varepsilon_i,\quad \varepsilon_i \sim \mathcal{N}\bigl(0,\sigma^2\bigr),\ \ i=1,\ldots,N, \quad\mbox{i.i.d.},
\]
where $b=(b_1,\ldots,b_n)^\top$ is an unknown coefficient vector, and
$f(x)=(f_1(x),\ldots,f_n(x))^\top$, $x\in\mathcal{X}\subset\R^n$, is a piecewise smooth regression basis vector.
The ordinary least square estimator (OLS) $\widehat b$ of $b$ is distributed as $\mathcal{N}_n(b,\sigma^2\Sigma)$,
where $\Sigma = \Bigl(\sum_{i=1}^N f(x_i) f(x_i)^\top\Bigr)^{-1}$.
We assume that $\sigma^2$ is unknown but its estimator $\widehat\sigma^2$ independent of $\widehat b$ is available.

The standard form of the $100(1-\gamma)\%$ simultaneous confidence band of hyperbolic-type is 
\begin{equation}
\label{hyperbolic}
 b^\top f(x) \in \widehat b^\top f(x) \pm c_\gamma \widehat\sigma \sqrt{f(x)^\top\Sigma f(x)},
\end{equation}
where $u\pm v$ denotes the region $(u-v,u+v)$.
The threshold $c_\gamma$ is determined so that the event (\ref{hyperbolic}) holds for all $x\in\mathcal{X}$ with given probability $1-\gamma$ \cite{liu:2010}.
That is,
\begin{equation}
\label{c_alpha}
 \gamma = \Pr\left(\sup_{x\in \mathcal{X}}\left|\frac{(\widehat b-b)^\top f(x)}{\widehat\sigma\sqrt{f(x)^\top\Sigma f(x)}}\right| \ge c_\gamma\right).
\end{equation}
This probability is equal to the right-hand side of (\ref{sup_Xi}) with $c=c_\gamma$ by setting
\[
 \zeta = \sigma^{-1}\Sigma^{-1/2}(\widehat b-b), \qquad
 M = \left\{\pm\frac{\Sigma^{1/2}f(x)}{\sqrt{f(x)^\top\Sigma f(x)}} \mid x\in\mathcal{X} \right\}, \qquad
 s := \widehat\sigma/\sigma.
\]
Although the tube method is used to approximate (\ref{c_alpha}) even when $\sigma^2$ is unknown (see, e.g., \cite{sun-loader:1994}),
this approximation is not yet well justified.

\section{Preliminaries}
\label{sec:pre}

In this section, we explain the notation used and state some regularity assumptions on the set $M$ together with the expressions for the exact tail probability (\ref{sup_Xi}) and the tube formula (\ref{Ptube0}).

\subsection{Notation}

Let $\bar\R=\R\cup \{\pm\infty\}$ be the extended real line and $\bar \R_+=(0,+\infty)\cup\{+\infty\}$.
For two functions $f,g:\R_+\to\R_+$, we write $f(y)\sim g(y)$ as $y\to +\infty$ iff $\lim_{y\to +\infty}f(y)/g(y)=1$, and  $f(y)\asymp g(y)$ as $y\to +\infty$ iff $\lim_{y\to +\infty}f(y)/g(y)\in (0,+\infty)$.
Let $\sharp A$ be the cardinality of a set $A$.

Let $r>0$ be a positive random variable, and let $\eta=(\eta_1,\ldots,\eta_n)\in\R^n$, $n\ge 2$, be a random vector uniformly distributed on the unit sphere $\S^{n-1}$.
We assume that $r$ and $\eta$ are stochastically independent.
The $n$-dimensional random vector $\xi=(\xi_1,\ldots,\xi_n) = r\eta$ is said to have a \emph{spherically contoured} distribution
(i.e., an elliptically contoured distribution with 
identity dispersion matrix).
For instance, if $r^2\sim\chi^2_n$ is $\chi^2$-distributed with $n$ degrees of freedom then $\xi\sim\mathcal{N}_n(0,I_n)$ is $n$-variate standard normal.
If $r^2\sim\nu\chi^2_n/\chi^2_\nu$ (the distribution of $\nu$ times the ratio of two independent chi-square random variables), $n,\nu\in\N$, then $\xi\sim t_\nu$ is distributed as the $n$-variate Student distribution with $\nu$ degrees of freedom.

Using $r$ and $\eta$, the spherically contoured random field was defined by (\ref{Xi_u}).
Note that for independent $\zeta\sim \mathcal{N}_n(0,I_n)$ and $s$, $\zeta/s$ has a spherically contoured distribution 
due to  the decomposition
\begin{equation*}
 \xi = \zeta/s = r \eta \quad\mbox{with}\ \ %
 r = \Vert\zeta\Vert/s, \quad
 \eta = \zeta/\Vert\zeta\Vert,
\end{equation*}
where $r$ and $\eta\sim\mathrm{Unif}(\S^{n-1})$ are independent.
Therefore, $\Xi_u$ defined in (\ref{Xi_u-again}) is a spherically contoured random field.
Introduce the notation $R_k=\sum_{i=1}^k \xi_i^2$ and $\bar R_{n-k}=R_n - R_k$, $k=1,\ldots,n$.
Notice that $R_n=r^2=\Vert\xi\Vert^2$.

Let $B(x;p,q)=\int\limits_0^x y^{p-1} (1-y)^{q-1}\,\dd y$,
$x\in[0,1]$ be the incomplete Beta function, where $p,\,q>0$
and $B(p,q)=B(1;p,q)$ is the usual Beta function.
Denote by $B_{p,q}$ the Beta distribution with parameters $p,q>0$.

\subsection{Critical radius and regularity conditions}

To state the regularity conditions on the domain $M\subset\S^{n-1}$, we first define the local critical radius (local reach) $\theta(x,v)$ of $M$ at a point $x$ in normal direction $v$.
For $x\in M\subset\S^{n-1}$, let $T_x\S^{n-1}$ be the tangent space of $\S^{n-1}$ at $x$, that is, the set of vectors orthogonal to $x$.
Let $N_x M$ be the normal cone of $M$ at $x$ restricted to $T_x\S^{n-1}$.
We set $N_xM=\{0\}$ for all interior points $x$ of $M$.
Let $S(N_x M)$ be the set of unit vectors in $N_x M$.
A point $y\in M^c\cap\S^{n-1}$ is called \emph{projection unique} if the nearest point (in the geodesic distance) $x\in M$ from $y$ is uniquely determined.

For $x\in M$ and $v\in S(N_x M)$,
the exponential map $S(N_x M)\to M$ is denoted by
\[
 v \mapsto \exp_x(\theta v), \ \ \theta\ge 0,
\]
where
\[
 \exp_x(\theta v) = x \cos\theta + v \sin\theta.
\]
Then, define
\[
 \theta(x,v) := \sup \{\theta\ge 0 \mid \exp_x(\theta v) \mbox{ is projection unique} \}.
\]
An explicit formula for $\theta(x,v)$ is given in \cite[Section\ 2.3]{kuriki-takemura:2009} as
\[
 \cot\theta(x,v) = \sup_{y\in M\setminus\{x\}}\frac{\langle v,y\rangle}{1-\langle x,y\rangle}.
\]
Its derivation for a finite set $M$ is given in Section \ref{sec:bonferroni}.
Informally speaking, $\theta(x,v)$ is the minimal arc length on the sphere such that all points closer to $x\in \partial M$ in normal direction $v$ than $\theta(x,v)$ have a unique metric projection $x$ onto the boundary $\partial M$ of $M$.

Then, the \emph{critical radius} or \emph{reach} is the infimum of all possible $\theta(x,v)$:
\[
 \theta_* = \inf_{x\in M}\inf_{v\in S(N_x M)}\theta(x,v).
\]
As we will see, $\theta_*$ characterizes the asymptotic behavior of $\Delta(c)$.

We say that $M$ is a \emph{stratified $C^2$-manifold of dimension} $m\le n-1$ if
$M=\bigsqcup_{d=0}^m \partial M_d$, where $\partial M_d$ is a $d$-dimensional $C^2$-manifold (without boundary) and $\bigsqcup$ denotes the disjoint union of sets.

Let $H_d(x,v)$ be the second fundamental form of $\partial M_d$ at a point $x\in\partial M_d$ with normal direction $v\in S(N_x M)$ with respect to an orthonormal basis.
There is a sign ambiguity in the definition of the second fundamental form.
In this paper, $H_d(x,v)$ is defined to be nonnegative definite if $M$ is geodesically convex. 

In the sequel, we impose the following
\begin{assumption}
\label{as:M}
\begin{enumerate}[label={\rm (\roman*)}]

\item
$M$ is a stratified $C^2$-manifold of dimension $m\le n-1$.

\item
$M$ is of positive reach, that is, $\theta_*>0$.

\end{enumerate}
\end{assumption}

\subsection{Exact tail probability and tube formula}

Following \cite{takemura-kuriki:2003}, let us introduce the fibre space
\[
 \mathcal{X}_d := \bigsqcup_{y\in\partial M_d}\Bigl(\{y\}\times S(N_y M)\Bigr), \quad 0\le d\le m,\quad\mbox{and}\quad \mathcal{X} := \bigsqcup_{d=0}^m \mathcal{X}_d,
\]
where $\partial M_d$ is the boundary of the $d$-dimensional component of $M$.
We also define
\[
 w_{k}(x,v) := \frac{1}{\Omega_{k}\Omega_{n-k}}\sum_{d=k-1}^m \1_{\{x\in\partial M_d\}}\tr_{d-k+1} H_d(x,v), \quad 1\le k\le m+1,
\]
where $\1_{\{\cdot\}}$ is the indicator function, 
$\tr_j H_d(x,v)$ is the $j$th elementary symmetric function of the eigenvalues of $H_d$,
and $\Omega_k = 2\pi^{k/2}/\Gamma(k/2)$ is the surface area of the unit sphere $\S^{k-1}$.
Let $\dd x$ be the surface area element of $\partial M_d$, and $\dd v$ be the surface area element of $S(N_x M)$.
Throughout the paper, we require an additional assumption.
\begin{assumption}
Functions $w_k(x,v)$ are absolutely integrable on $\mathcal{X}$ with respect to the measure $\dd x\,\dd v$.
\end{assumption}
It is satisfied, for example, if all eigenvalues of $H_d$ are uniformly bounded on $\mathcal{X}$.
Recall that
\[
 M_c = \{u\in M\mid \Xi_u=\langle u,\xi\rangle \ge c\}
\]
is the excursion set of the spherically contoured random field $\Xi_u$ on $u\in M$ at a level $c>0$.
It is obviously a random closed set.
Recall also that
\[
 P(c) = \Pr\biggl(\sup_{u\in M}\Xi_u \ge c\biggr) \quad\mbox{and}\quad
 P_\mathrm{tube}(c) = \E[\chi(M_c)].
\]
The relation \eqref{tube} below can be proven similarly to \cite[Proposition 3.2]{takemura-kuriki:2003}, while relation \eqref{max} can be found in \cite[Theorem 2.2]{takemura-kuriki:2002}:
\begin{proposition}
Let $\Xi$ be the random field defined in (\ref{Xi_u}) or (\ref{Xi_u-again}).
Assume that the set $M$ satisfies Assumption \ref{as:M}.
Then, for any $c>0$, it holds
\begin{align}
 P_\mathrm{tube}(c) =& \sum_{k=1}^{m+1} \int_{\mathcal{X}} w_{k}(x,v)\, \dd x\, \dd v \cdot \Pr\bigl(R_{k}\ge c^2\bigr),
\label{tube}
\\
 P(c) =& \sum_{k=1}^{m+1} \int_{\mathcal{X}} w_{k}(x,v)\,\Pr\bigl(R_{k}\ge c^2, \bar R_{n-k}\le t(x,v) R_{k}\bigr)\, \dd x\, \dd v,
\label{max}
\end{align}
where $t(x,v)=\tan^2\theta(x,v)>0$, $(x,v)\in \mathcal{X}$.
\end{proposition}

\subsection{Bonferroni approximation}
\label{sec:bonferroni}

As a simplest example, consider the case of a finite set $M$
which appears in the multiple testing problem introduced in Section \ref{sec:examples}.
Let $M=\{u_1,\ldots,u_N\}$ be a finite set of $N$ points on $\S^{n-1}$.
Let $T_i$ be the $i$th test statistic, $i=1,\ldots, N$.
Under the null hypothesis $\E[T_i]\equiv 0$, the $T_i$ are stochastically represented as $\langle u_i,\xi\rangle$,
where $\xi$ is a spherically contoured random vector.

Without loss of generality, fix $u_1=e_1=(1,0,\ldots,0)$.
Then, noting that $w_{1}(x,v)=1/(\Omega_1\Omega_{n-1})$ and $H_0(x,v)=1$ for $x\in M$, $v\in \S^{n-2}$ with $m=0$, we write
\begin{equation*}
 P_\mathrm{tube}(c)
 = \frac{1}{\Omega_1\Omega_{n-1}} \sum_{x\in M}\int_{\S^{n-2}} \Pr\bigl(R_{1}\ge c^2 \bigr)\,\dd v
 = \frac{1}{2} \sum_{x\in M} \Pr\bigl(R_{1}\ge c^2\bigr)
 = N \Pr(\xi_1 \ge c)
\end{equation*}
as the first order Bonferroni approximation.

The exact probability of the maximum can be evaluated directly as
\begin{align*}
 P(c)
 =& \Pr\biggl(\max_{i=1,\ldots,N}\langle u_i,\xi\rangle \ge c\biggr) \\
 =& \sum_{i=1}^N \Pr\Bigl(\langle u_i,\xi\rangle \ge c,\ \langle u_i,\xi\rangle\ge \langle u_j,\xi\rangle,\,\forall j\ne i\Bigr) \\
 =& \sum_{i=1}^N \Pr\biggl(\langle u_i,\xi\rangle \ge c,\ \langle u_i,\xi\rangle \ge \frac{\langle u_j,\xi\rangle-\langle u_i,u_j\rangle\langle u_i,\xi\rangle}{1-\langle u_i,u_j\rangle},\,\forall j\ne i\biggr).
\end{align*}

Let us examine the first term with $i=1$.
Recall that $\langle u_1,\xi\rangle=\xi_1$.
Put
\[
 v := \frac{(0,\xi_2,\ldots,\xi_n)}{\sqrt{\bar R_{n-1}}}, \quad
 \bar R_{n-1}=\xi_2^2+\cdots+\xi_n^2.
\]
Then $v$ is a unit random vector in the normal space $N_{e_1}M=\{v \mid \langle e_1,v\rangle =0\}$ at $e_1$, and
$\xi$ can be represented as
$\xi = \xi_1 e_1 + \sqrt{\bar R_{n-1}}v$.
Taking this into account, we have
\begin{align*}
 \Pr & \biggl(\langle u_1,\xi\rangle \ge c,\ \langle u_1,\xi\rangle\ge \frac{\langle u_j,\xi\rangle-\langle u_1,u_j\rangle\langle u_1,\xi\rangle}{1-\langle u_1,u_j\rangle},\,\forall j>1 \biggr) \\
 &= \Pr\biggl(\xi_1 \ge c,\ \xi_1\ge \frac{\langle u_j,\sqrt{\bar R_{n-1}}v\rangle}{1-\langle e_1,u_j\rangle},\,\forall j>1 \biggr) \\
 &= \Pr\biggl(\xi_1 \ge c,\ \xi_1\ge \sqrt{\bar R_{n-1}}\max_{j>1}\frac{\langle u_j,v\rangle}{1-\langle e_1,u_j\rangle}\biggr) \\
 &= \frac{1}{2}\,\Pr\bigl(R_1 \ge c^2, R_1\ge \bar R_{n-1}\tan^{-2}\theta(u_1,v)\bigr),
\end{align*}
where 
\[
 \tan\theta(u_i,v) := \biggl(\max_{j\ne i}\frac{\langle u_j,v\rangle}{1-\langle u_i,u_j\rangle}\biggr)^{-1},\quad v\in S(N_{u_i}M), \ \ i=1,\ldots, N.
\]
This leads to
\begin{equation*}
 P(c)
 = \frac{1}{2\Omega_{n-1}} \sum_{u\in M}\int_{\S^{n-2}} \Pr\bigl(R_{1}\ge c^2, \bar R_{n-1}\le \tan^{2}\theta(u,v) R_{1}\bigr)\, \dd v,
\end{equation*}
which corresponds to relation \eqref{max}.
As will be shown in Section \ref{subsec:Bonferr},
$P_\mathrm{tube}(c)$ and $P(c)$ are asymptotically equivalent for $c\to +\infty$
except for the case of a regularly varying tail.
Noting that $\max_{v\in T_{u_i}\S^{n-1}}\langle v,u_j \rangle = \sqrt{1-\langle u_j,u_i\rangle^2}$, we have
\begin{align*}
 \tan\theta_* & =
 \min_i \min_{v\in S(N_{u_i}M)} \tan\theta(u_i,v)
 = \left(\max_i \max_{j\ne i}\frac{\sqrt{1-\langle u_j,u_i\rangle^2}}{1-\langle u_i,u_j\rangle}\right)^{-1} \\
 & = \min_{i\ne j} \sqrt{\frac{1-\langle u_i,u_j\rangle}{1+\langle u_i,u_j\rangle}}
  = \sqrt{\frac{1-\rho_*}{1+\rho_*}}, \qquad \rho_*=\max_{i\neq j}\rho_{ij},
\end{align*}
\begin{equation}
\label{theta*}
 \cos\theta_* = \sqrt{\frac{1+\rho_*}{2}}.
\end{equation}

\begin{remark}
\label{rem:voronoi}
The great circle distance between $u,v\in\S^{n-1}$
is given by $\cos^{-1}\langle u,v\rangle$.
The spherical Voronoi diagram generated by the points $\{u_i\}_{1\le i\le N}$ is the family of Voronoi cells
\[
 D_i=\{u\in\S^{n-1}\mid \langle u,u_i\rangle\ge \langle u,u_j\rangle, \ \forall j\ne i\}, \ \ i=1,\ldots,N.
\]
Let $\exp_u(\psi v)$, $\psi\ge 0$, be the exponential map at $u$.
Then, $\theta(u_1,v)=\cos^{-1}\langle u_1,u'\rangle$,
where $u'$ is the crossing point of the orbit $\{\exp_{u_1}(\psi v) \mid \psi\ge 0\}$ and the cell boundary of $D_1$
(see Figure \ref{fig:voronoi}, left).
If $u_2$ is the nearest neighbor of $u_1$, the minimum of the function $v\mapsto\theta(u_1,v)$ is attained at $v=(u_2-u_1)/\Vert u_2-u_1\Vert$, and
\[
 \min_v\theta(u_1,v) = \cos^{-1}\langle u_1,u''\rangle = \cos^{-1}\langle u_2,u''\rangle =
 \cos^{-1}\sqrt{\frac{1+\langle u_1,u_2\rangle}{2}}
\]
with $u''=(u_1+u_2)/\Vert u_1+u_2\Vert$ (see Figure \ref{fig:voronoi}, right).
\end{remark}

\begin{figure}[ht]
\begin{center}
\begin{tabular}{ccc}
\scalebox{0.60}{\includegraphics{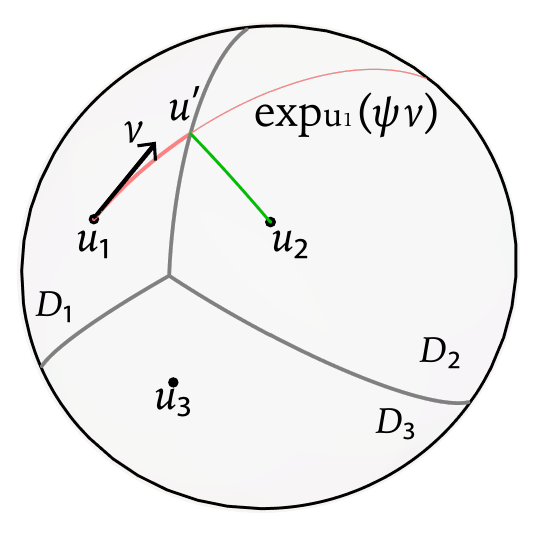}} &&
\scalebox{0.60}{\includegraphics{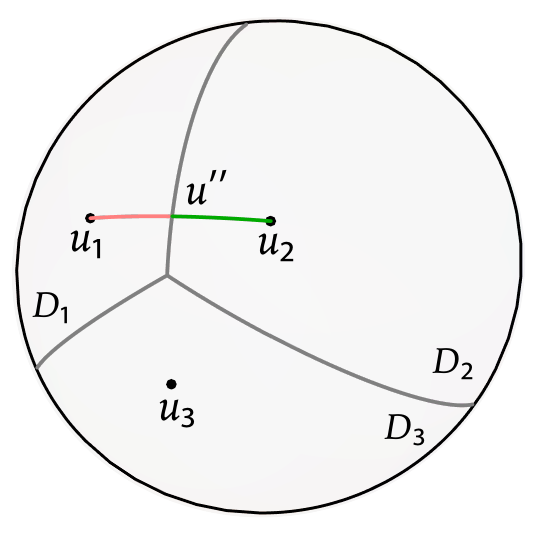}}
\end{tabular}
\caption{Voronoi cells $D_i$ on the unit sphere, distance $\theta(u,v)$ and its minimum $\min_v\theta(u,v)$.}
Left: $\theta(u_1,v)=\cos^{-1}\langle u_1,u'\rangle$.
Right: $\min_v\theta(u_1,v)$ is attained at $v=\frac{u_2-u_1}{\Vert u_2-u_1\Vert}$.
\label{fig:voronoi}
\end{center}
\end{figure}

\section{Classes of heavy-tailed distributions}
\label{sec:heavy-tailed}

In the sequel, we assume that the distribution of the random variable $R_n=r^2=\Vert\xi\Vert^2$ may exhibit different types of tail behavior.
For this purpose, we introduce a number of distribution classes, particularly heavy-tailed ones
\cite{goldie-kluppelberg:1998,foss-etal:2013}.

Let $F:\R_+\to [0,1]$ be a cumulative distribution function
with $F(x)<1$ for all $x>0$, and let $\bar F:=1-F$ be its tail distribution function.
\begin{definition}
\label{def:H}
$F$ is called heavy-tailed, or $F\in\mathcal{H}$, whenever
$\int_0^{+\infty} e^{\lambda x}\dd F(x) = \infty$ for all $\lambda>0$.
This condition is equivalent to
$\limsup_{x\to +\infty} e^{\lambda x}\bar F(x) = +\infty$ for all $\lambda>0$.
$F$ is called light-tailed, or $F\in\mathcal{H}^c$, if
there exists a $\lambda>0$ such that
$ \int_0^{+\infty} e^{\lambda x}\dd F(x) < +\infty$.
\end{definition}
\begin{definition}
$F$ is called long-tailed, or $F\in\mathcal{L}$, if 
\begin{equation*}
 \lim_{x\to +\infty} \frac{\bar F(x-y)}{\bar F(x)} = 1, \quad y>0.
\end{equation*}
\end{definition}
\begin{definition}
$F$ is called subexponential, or $F\in\mathcal{S}$, if
$\overline{F*F}(x) \sim 2 \bar F(x)$ as ${x\to +\infty}$,
where $F*F$ is the convolution of $F$ with itself.
\end{definition}
\begin{definition}
$F$ is called regularly varying with index $-\gamma >-\infty$, or $F\in\mathrm{RV}_{-\gamma}$, if and only if
\begin{equation*}
 \lim_{x\to +\infty} \frac{\bar F(\lambda x)}{\bar F(x)} = \frac{1}{\lambda^\gamma}\quad\mbox{for all }\ \lambda>0.
\end{equation*}
$F\in\mathrm{RV}_0$ is called \emph{slowly varying}.
Let $\mathrm{RV}=\bigcup_{\gamma\ge 0}\mathrm{RV}_{-\gamma}$.
\end{definition}

\begin{definition}
$F$ belongs to the class $\mathcal{S}^*$ if and only if the mean
$\mu=\int_0^{+\infty} \bar F(x) \dd x<+\infty$ exists and
\begin{equation*}
 \lim_{x\to +\infty} \int_0^x \frac{\bar F(x-y)}{\bar F(x)} \bar F(y)\,\dd y= 2 \mu.
\end{equation*}
\end{definition}
Many known distributions belong to the class $\mathcal{S}^*$ as long as their mean exists, for instance, the Pareto, log-normal, log-gamma, and Weibull distributions.

\begin{proposition}
It holds that
$\mathrm{RV}_{-\gamma} \subset \mathcal{S} \subset \mathcal{L} \subset \mathcal{H}$
and $\mathcal{S}^*\subset\mathcal{S}$.
\end{proposition}

\begin{proof}
See \cite[Theorem 2]{cistjakov:1964} for $\mathcal{S} \subset \mathcal{L} \subset \mathcal{H}$ as well as
\cite{goldie-kluppelberg:1998} for $\mathcal{S}^*\subset\mathcal{S}$ and
$\mathrm{RV}_{-\gamma} \subset \mathcal{S}$.
\end{proof}

The following fact is shown in \cite[Lemma 3.4]{goldie-kluppelberg:1998}:
\begin{proposition}
\label{prop:L}
A distribution function $F$ is tong-tailed, or $F\in\mathcal{L}$, if and only if $\bar F$ admits a representation
\begin{equation}
\label{eq:FL}
 \bar F(x)=C(x) \exp\left\{-\int_{x_0}^x q(t)\, \dd t\right\}, \quad x\ge x_0,
\end{equation}
for some $x_0\ge 0$, functions $C,\,q: \R_+\to \R_+$ such that $C(x)\to C_0>0$, $q(x)\to 0$ as $x\to +\infty$, $\int_{x_0}^{+\infty} q(t)\,\dd t=+\infty$.
\end{proposition}

Keeping the form (\ref{eq:FL}) in mind, we introduce the class $\mathcal{L}_{\beta,\gamma}$ of distributions indexed by $(\beta,\gamma)$.
\begin{definition}
\label{def:classL_ab}
The distribution function $F$ is said to belong to $\mathcal{L}_{\beta,\gamma}$ if its tail probability function $\bar F$ is asymptotically of the form
\[
 \bar F(x) \sim C(x) \exp\left\{-\int_{x_0}^x \frac{\ell(t)}{t^{\beta}} \dd t \right\}, \quad x\to +\infty,
\]
for some $x_0\ge 0$, where $\beta\le 1$, $C(x)\to C$, and $\ell(\cdot)$ is a slowly varying function such that
\[
 \lim_{t\to +\infty}\ell(t)=\gamma\in \bar R_+.
\]
Write
\[
 \mathcal{L}_{*,*} = \bigcup_{\beta\le 1}\bigcup_{\gamma\in \bar\R_+}
 \mathcal{L}_{\beta,\gamma}.
\]
\end{definition}

\begin{proposition}
\label{prop:classL_ab}
Let $F \in \mathcal{L}_{\beta,\gamma}$.
Then 
\begin{enumerate}[label={\rm (\roman*)}]
\item\label{Prop4.i}
 $F$ belongs to $\mathcal{L}$ if $\beta\in (0,1]$.

\item
$F$ belongs to $\mathcal{H}^c$ (the compliment of $\mathcal{H}$) if and only if $\beta\le 0$.

\item\label{Prop4.iii}
$F$ belongs to $\mathrm{RV}$ if and only if $\beta=1$ and $\gamma\in [0,+\infty)$.
In this case, $ F\in\mathrm{RV}_{-\gamma}$.

\item\label{Prop4.iv}
$F$ belongs to $\mathcal{S}$ if and only if $\beta\in(0,1]$. 
\end{enumerate}
\end{proposition}

\begin{proof}
Put $q(t)=\ell(t)/t^\beta$, $t>0$.
\begin{enumerate}[label={\rm (\roman*)}]
\item
When $\beta<1$,
\[
\int_{x_0}^x q(t)\dd t=\int_{x_0}^x \ell(t)/t^\beta \dd t\sim (1-\beta)^{-1} \ell(x) x^{1-\beta}
\]
by Karamata's theorem (e.g., \cite[Theorem 1.5.11]{bingham-etal:1987}, \cite[Theorem 2.45]{foss-etal:2013}), which tends to infinity as $x\to +\infty$.
When $\beta=1$ and for any $\varepsilon>0$, we have
\[
 \int_{x_0}^x q(t)\dd t=\int_{x_0}^x \ell(t)/t\,\dd t \ge (\gamma-\varepsilon) \int_{x_0}^x \dd t/t \to +\infty \quad\mbox{as}\ \ x\to +\infty.
\]
Moreover, for $\beta>0$ it holds $q(x) \to 0$ as $x\to +\infty$.
Then 
$F\in\mathcal{L}$ if $\beta\in (0,1]$ by Proposition \ref{prop:L}.

\item
Set $\lambda_0\in(0,\gamma ) $.
When $\beta\le 0$, for $q(t)=\ell(t)/t^\beta$, it holds that $q(t) > \lambda_0$ for $t$ large enough, and consequently, for $x$ large enough,
\[
 \int_{x_0}^x q(t) \dd t > (x-x_0) \lambda_0.
\]
Thus,
\[
 \limsup_{x\to +\infty} e^{\lambda x}\bar F(x) \le 
 C\limsup_{x\to +\infty} e^{\lambda x -\lambda_0(x-x_0)} \dd x <+\infty
\]
for $\lambda\in(0,\lambda_0)$.
Then, $F\in\mathcal{H}^c$ by Definition \ref{def:H}.
When $\beta>0$, it holds $F\in\mathcal{L}\subset\mathcal{H}$ by \ref{Prop4.i}, and hence $F\notin\mathcal{H}^c$.

\item
According to \cite[Theorem 1.3.1]{bingham-etal:1987}, any slowly varying function has the form
\[
 C(x)\exp\left\{-\int_{x_0}^x \frac{\varepsilon(t)}{t}\,\dd t\right\},\quad \varepsilon(x)\to 0,
\]
where $C(\cdot)$ is measurable with $C(x)\to C\in(0,+\infty)$ as $x\to +\infty$.
If $F\in \mathrm{RV}$ then $F\in \mathrm{RV}_{-\gamma}$ for some $0\le \gamma<+\infty$, and hence
\begin{align*}
 \bar F(x) = x^{-\gamma}C(x)\exp\left\{-\int_{x_0}^x \frac{\varepsilon(t)}{t}\,\dd t\right\}
= C_1(x)\exp\left\{-\int_{x_0}^x \frac{\varepsilon_1(t)}{t}\,\dd t\right\}, 
\end{align*}
where
$\varepsilon_1(x):=\varepsilon(x)+\gamma\to \gamma$ as $x\to +\infty$ and $C_1(x):=C(x) x_0^\gamma$.
By Definition \ref{def:classL_ab}, it holds $F\in \mathcal{L}_{1,\gamma}$.
Conversely, if $F\in \mathcal{L}_{1,\gamma}$ for some $\gamma\in[0,+\infty)$, the very same reasoning as above (read bottom up) yields $F\in \mathrm{RV}_{-\gamma}$.

\item
For $q(x)=\ell(x)/x^\beta$ with $\beta\in(0,1)$,
or $(\beta,\gamma)=(1,+\infty)$,
we have $x q(x)\to +\infty$, $q(x)\to 0$, $q\in\mathrm{RV}_{-\beta}$, hence by \cite[Corollary 3.9 (b)]{goldie-kluppelberg:1998}, $F\in\mathcal{S}^*\subset\mathcal{S}$ holds.
If $q(x)=\ell(x)/x$ with $\lim_{x\to +\infty}\ell(x)=\gamma<+\infty$, then $F\in\mathrm{RV}\subset\mathcal{S}$ by \ref{Prop4.iii}.
When $q(x)=\ell(x)/x^\beta$ with $\beta\le 0$, we have
$F\in\mathcal{H}^c\subset\mathcal{S}^c$.
\end{enumerate}
\end{proof}



\begin{example}
\label{ex:L}
Many widely used probability laws belong to the class $\mathcal{L}_{*,*}$.

\begin{enumerate}[label={\rm (\roman*)}]
\item
Log-normal distribution $\exp\{\mathcal{N}(0,1)\}$ with density
\[
 \frac{1}{\sqrt{2\pi}x}\exp\left\{-\frac{(\log x)^2}{2}\right\}, \quad x>0,
\]
and tail probability function
\begin{equation}
\label{log_normal}
\begin{aligned}
 \bar F(x)
&\sim \frac{1}{\sqrt{2\pi}\log x}\exp\left\{-\frac{(\log x)^2}{2}\right\} \\
&\sim \exp\left\{-\int_c^x \left(\frac{\log t}{t}+\frac{1}{t\log t}\right)\,\dd t\right\},\quad x\to +\infty,
\end{aligned}
\end{equation}
belongs to $\mathcal{L}_{1,+\infty}$ with $\ell(x)=\log x+1/\log x$.
Here and in what follows, $c$ is an appropriate constant which will not be specified.
By Proposition \ref{prop:classL_ab} \ref{Prop4.iii}--\ref{Prop4.iv}, it is subexponential but not regularly varying.

\item
$\chi^2_\nu$ distribution with $\nu>0$ degrees of freedom with the tail probability function
\begin{equation}
\label{chi2}
\bar F(x) \sim \exp\left\{-\int_c^x \left(\frac{1}{2}-\frac{\nu-2}{2 t}\right)\,\dd t\right\}, \quad x\to +\infty,
\end{equation}
belongs to $\mathcal{L}_{0,1/2}$ with $\ell(x)=(1/2)-(\nu-2)/(2x)$.
It is clearly light-tailed.

\item
$\chi_\nu$ distribution with $\nu>0$ degrees of freedom (defined by $\chi_\nu:=\sqrt{\chi^2_\nu}$) and the tail probability
\begin{equation*}
\bar F(x) \sim \exp\left\{-\int_c^x \left(t-\frac{\nu-2}{t}\right)\,\dd t\right\}, \quad x\to +\infty,
\end{equation*}
belongs to $\mathcal{L}_{-1,1}$ with $\ell(x)=1-(\nu-2)/x^2$.
It is also light-tailed.

\item
$F_{\nu_1,\nu_2}$-distribution with $\nu_1,\nu_2>0$ degrees of freedom and the tail probability function
\begin{equation*}
\begin{aligned}
 \bar F(x)
 &\sim \frac{\Gamma(\frac{\nu_1+\nu_2}{2})}{\Gamma(\frac{\nu_1}{2})\Gamma(\frac{\nu_2}{2}+1)}
\left(\frac{\nu_1 x}{\nu_1 x+\nu_2}\right)^{\nu_1/2}
\left(1-\frac{\nu_1 x}{\nu_1 x+\nu_2}\right)^{\nu_2/2} \\
 &\sim \exp\left\{-\int_c^x \left(\frac{\nu_1\nu_2}{2(\nu_1 t+\nu_2)}\right)\,\dd t\right\},\quad x\to +\infty,
\end{aligned}
\end{equation*}
belongs to $\mathcal{L}_{1,\nu_2/2}$ with $\ell(x)=\frac{\nu_2}{2(1+(\nu_2/\nu_1) x^{-1})}$, hence $\mathrm{RV}_{-\nu_2/2}$ by Proposition \ref{prop:classL_ab} \ref{Prop4.iii}.

\item
Bessel distribution with degrees of freedom $\nu_1,\nu_2>0$ is given by the density
\[
 \frac1{2^{\frac{\nu_1+\nu_2-2}{2}}\Gamma(\frac{\nu_1}{2})\Gamma(\frac{\nu_2}{2})} x^{\frac{\nu_1+\nu_2-4}{4}} K_{\frac{\nu_1-\nu_2}{2}}(\sqrt{x}), \quad x>0,
\]
where $K_j(\cdot)$ is the modified Bessel function of the second kind.
It is the probability law of the product of independent $\chi_{\nu_1}^2$- and $\chi_{\nu 2}^2$-distributed random variables with degrees of freedom $\nu_1,\nu_2>0$ (see \cite{wells-etal:1962}).
Its tail probability function is given by
\begin{equation}
\begin{aligned}
\label{Bessel}
 \bar F(x)
 &\sim \frac{\sqrt{\pi}}{2^{\frac{\nu_1+\nu_2-3}{2}}\Gamma(\frac{\nu_1}{2})\Gamma(\frac{\nu_2}{2})} x^{\frac{\nu_1+\nu_2-3}{4}} \exp\bigl(-\sqrt{x}\bigr) \\
 &\sim \exp\left\{-\int_c^x \left(\frac{1}{2\sqrt{t}}-\frac{\nu_1+\nu_2-3}{4 t}\right)\,\dd t\right\}, \quad x\to +\infty.
\end{aligned}
\end{equation}
Thus, Bessel distribution belongs to $\mathcal{L}_{1/2,1/2}$ with $\ell(x)=1/2-(\nu_1+\nu_2-3)/(4\sqrt{x})$.

It is subexponential but not regularly varying.
\end{enumerate}
\end{example}

The relationships between the distribution classes $\mathcal{L}_{\beta,\gamma}$, $\mathcal{H}$, $\mathcal{S}$ and $\mathrm{RV}$ illustrated in Example \ref{ex:L} are summarized in Table \ref{table:class1}.
Here, $\mathcal{S}\setminus\mathrm{RV}$ means $\mathcal{S}\cap\mathrm{RV}^c$.

\begin{table}[ht]
\caption{Class $\mathcal{L}_{\beta,\gamma}$ and representative distributions.}
\label{table:class1}
\begin{center}
\begin{tabular}{ccccccc@{}}
\toprule
Tail class
& $\mathcal{H}^c$ & $\mathcal{H}^c$ & $\mathcal{S}\setminus\mathrm{RV}$ & $\mathcal{S}\setminus\mathrm{RV}$ & $\mathrm{RV}_{-\gamma}$ \\
\midrule
\multirow{2}{*}{$\mathcal{L}_{\beta,\gamma}$} & \multirow{2}{*}{$\beta<0$} & \multirow{2}{*}{$\beta=0$} & \multirow{2}{*}{$\beta\in(0,1)$} & $\beta=1$ & $\beta=1$ \\
& & & & $\gamma=+\infty$ & $\gamma<+\infty$ \\[2pt] 
\midrule
Example & $\chi$ & $\chi^2$ & Bessel & Log-normal & $F$ 
\\
$(\beta,\gamma)$ & $(-1,1)$ & $(0,\frac12)$ & $(\frac12,\frac12)$ & $(1,+\infty)$ & $(1,\gamma)$ \\[1pt]
\bottomrule
\end{tabular}
\end{center}
\end{table}

\section{Main results}
\label{sec:Main}

We now state the main results of the paper.
Short proofs are given on the spot, while longer proofs are provided in Section \ref{sec:proof}.

\subsection{Asymptotic relative error of the tube method approximation}

Recall the notation
$P_\mathrm{tube}(c) = \E[\chi(M_c)]$ and $P(c) = \Pr\bigl(\sup_{u\in M} \langle u,\xi\rangle\ge c \bigr)$.
In the sequel, we evaluate the asymptotics of the relative error defined by
\[
 \Delta(c) = \frac{P_\mathrm{tube}(c) - P(c)}{P_\mathrm{tube}(c)}
\]
as $c\to +\infty$, and show that, in some important cases, it does not necessarily tend to zero.

Throughout this section, we deal with the spherically contoured random field $\Xi=\{\Xi_u, u\in M\}$ defined in (\ref{Xi_u}) under Assumption \ref{as:M} on its index domain $M\subset \S^{n-1}$.
Recall that $R_n=r^2=\sum_{j=1}^n \xi_j^2$.
For its distribution function $F_{R_n}(x)=\Pr(R_n\le x) $, $x\in\R_+$, we
now assume that $F_{R_n}\in \mathcal{L}_{\beta,\gamma}$ with $\beta\le 1$, $\gamma\in\bar\R_+$ and corresponding function $q(t)=\ell(t)t^{-\beta}$.
Let the random variable 
$B_k \sim B_{\frac{k}{2},\frac{n-k}{2}}$ be Beta distributed with parameters $k/2$ and $(n-k)/2$, $k=1,\ldots,m+1$, where $m=\dim M$.
Define the expression
\begin{equation}
\label{eq:D_k}
 D_k(\theta,c) := {\int_0^{\cos^2\theta} \exp\left\{ -\int_{c^2}^{c^2/B_k} q(t)\,\dd t \right\}\, \dd P_{B_k}}, \quad \theta\in[0,\pi/2],\ c>0,
\end{equation}
where $\dd P_{B_k} $ means integration with respect to the probability density of $B_k$.
Note that
\[
 \bigl(1+t(x,v)\bigr)^{-1} = \cos^2\theta(x,v).
\]
The proofs of Lemmas \ref{lem:MainL} and \ref{lem:MainD} below will be provided in Section \ref{sec:proof}.

\begin{lemma}
\label{lem:MainL}
For the above random field $\Xi$  with $F_{R_n}\in\mathcal{L}_{\beta,\gamma}$, it holds that
\begin{equation}
\label{eq:MainL}
 \Delta(c) \sim \frac{\sum_{k=1}^{m+1} \int_{\mathcal{X}} w_{k}(x,v) D_{k}\bigl(\theta(x,v),c\bigr)
 \, \dd x\, \dd v}{\sum_{k=1}^{m+1} D_{k}(0,c)\int_{\mathcal{X}} w_{k}(x,v)\, \dd x\, \dd v}, \quad c\to +\infty.
\end{equation}
\end{lemma}

Recall that $\beta\le 1$.
Introduce the function
\begin{equation}
\label{eq:g}
 g_\beta(y) :=
 \begin{cases} \displaystyle 
  \frac{y^{\beta-1} -1}{1-\beta} & \mbox{if }\beta<1, \\
  - \log y &\mbox{if } \beta=1,
  \end{cases}
 \qquad y\in [0,1].
\end{equation}
Notice that
$g'_\beta(y) = -y^{\beta-2}$ for all $\beta\le 1$.
Let $\ell_0(h)$ be an arbitrary function such that $\ell_0(h)\sim\ell(h)$ as $h\to +\infty$,
and define
\begin{equation}
\label{eq:rbeta_def}
 r_\beta(h,y) := h^{1-\beta} \int_1^{y^{-1}} \bigl\{\ell(hv) - \ell_0(h) \bigr\} v^{-\beta}\, \dd v.
\end{equation}
In Theorems \ref{thm:MainS*} and \ref{thm:MainS*_dim0}, the function $r_\beta$ accounts for higher-order terms in the asymptotics of the relative error $\Delta(c)$ as $c\to +\infty$.
Note that by \cite[Theorem 1.5.2]{bingham-etal:1987},
$\ell(h v)/\ell(h)=1+o(1)$ uniformly in $v$ on any compact interval as $h\to +\infty$,
which, combined with $\ell(h )=\ell_0(h)(1+o(1))$, yields by Karamata's theorem
 (e.g., \cite[Theorem 1.5.11]{bingham-etal:1987}, \cite[Theorem 2.45]{foss-etal:2013}) that
\begin{equation}
\label{eq:rbeta_def-o}
 r_\beta(h,y) = o\bigl(h^{1-\beta}\ell_0(h)\bigr) g_\beta(y), \quad h\to +\infty, \quad \mbox{{uniformly in $y$.}}
\end{equation}
Similarly,
\begin{align}
 \frac{\partial r_\beta(h,y)}{\partial y}
 &=
  -h^{1-\beta} \bigl\{\ell(h y^{-1}) - \ell_0(h) \bigr\} y^{\beta-2} \nonumber \\
 &= o\bigl(h^{1-\beta}\ell_0(h)\bigr) g'_\beta(y), \quad h\to +\infty, \quad \mbox{uniformly in $y$},
\label{eq:rbeta_def-o1}
\end{align}
which will be used later.

An appropriate choice of the function $\ell_0$ may significantly simplify the calculations.
In Example \ref{ex:three_dist}, $\ell_0(h)$ is chosen as the leading term of the asymptotic expansion of $\ell(h)$, and
$o(h^{1-\beta}\ell_0(h))$ in (\ref{eq:rbeta_def-o}) becomes $O(1)$ as $h\to +\infty$.

\begin{lemma}
\label{lem:MainD}
For $F_{R_n}\in\mathcal{L}_{\beta,\gamma}$,
$D_k(\theta,c)$ in (\ref{eq:D_k}) is asymptotically equivalent to
\begin{equation}
\label{eq:D_k_asy}
D_k(\theta,c) \sim
\begin{cases}
\displaystyle
\frac{ \cos^{k-2\beta+2}\theta\sin^{n-k-2}\theta\,
\exp\bigl\{ -c^{2(1-\beta)} \ell_0(c^2) g_\beta\left(\cos^2\theta\right) -r_\beta(c^2,\cos^2\theta) \bigr\}}{B\bigl(\frac{k}2,\frac{n-k}2\bigr) c^{2(1-\beta)} \ell_0(c^2)} \hspace*{-60mm} & \\ & \mbox{if }\beta<1 \mbox{ or } (\beta,\gamma)=(1,+\infty);\,\theta>0,\\[1.5em]
\displaystyle
 \frac{\Gamma\left(\frac{n-k}{2}\right)}{B\bigl(\frac{k}2,\frac{n-k}2\bigr) \{c^{2(1-\beta)} \ell_0(c^2)\}^{\frac{n-k}{2}}} & \mbox{if }\beta<1 \mbox{ or } (\beta,\gamma)=(1,+\infty);\,\theta=0, \\[1.5em]
\displaystyle
 a_{\gamma,k} \Pr\Bigl(\widetilde B_{\gamma,k}<\cos^2 \theta\Bigr) & \mbox{if } \beta=1,\gamma<+\infty
\end{cases}
\end{equation}
as $c\to +\infty$, where $\widetilde B_{\gamma,k}$ is a random variable 
distributed as the beta distribution $B_{\gamma+\frac{k}{2},\frac{n-k}{2}}$
and 
\begin{equation*}
a_{\gamma,k}:=\frac{B\bigl(\gamma+\frac{k}{2},\frac{n-k}{2} \bigr)}{B\bigl(\frac{k}{2},\frac{n-k}{2} \bigr)}.
\end{equation*}
\end{lemma}

Combining Lemmas \ref{lem:MainL} and \ref{lem:MainD}, we immediately obtain the statements of Theorems \ref{thm:MainRV} and \ref{thm:MainS*} below.
Theorem \ref{thm:MainRV} states that when $F_{R_n}$ is regularly varying and $w_{k}$ non-negative for all $k$,
the tube formula does not provide the correct leading term for $P(c)$ unless $\theta(x,v)\equiv \pi/2$ (i.e., $t(x,v)=+\infty$ and $M$ is geodesically convex).

\begin{theorem}
\label{thm:MainRV}
Suppose that $F_{R_n}\in\mathcal{L}_{1,\gamma}=\mathrm{RV}_{-\gamma}$, $\gamma\in(0,+\infty)$.
One has
\begin{equation}
\label{eq:MainRV}
 \Delta(c) \sim \frac{\sum_{k=1}^{m+1} a_{\gamma,{k}} \int_{\mathcal{X}} w_{k}(x,v) \Pr\Bigl(\widetilde B_{\gamma,{k}}<\cos^2 \theta(x,v)\Bigr)
\, \dd x\, \dd v}
{\sum_{k=1}^{m+1} a_{\gamma,{k}} \int_{\mathcal{X}} w_{k}(x,v)\, \dd x\, \dd v}
\end{equation}
as $c\to +\infty$, where the coefficients $a_{\gamma,k}$ and the random variables $\widetilde B_{\gamma,k}$ are defined in Lemma \ref{lem:MainD}.
\end{theorem}

\begin{proof}
Substitute the asymptotics for $D_k(\theta,c)$ as $c\to +\infty$ from (\ref{eq:D_k_asy}) into (\ref{eq:MainL}).
\end{proof}

The right-hand side of \eqref{eq:MainRV} is constant in $c$ and distinct from zero in many particular cases (for example, if functions $w_{k}$ are all positive and $\theta(x,v) \not\equiv \pi/2$).
Now give an upper bound for 
$|\Delta(c)|$:
\begin{remark}
\label{rem:bound0}
When $c$ is large, it holds that $|\Delta(c)| < \bar\Delta$, where
\begin{equation}
\label{bound0}
 \bar\Delta = \frac{\sum_{k=1}^{m+1} a_{\gamma,{k}} \int_{\mathcal{X}} w^+_{k}(x,v)\, \dd x\, \dd v\, \Pr\Bigl(\widetilde B_{\gamma,{k}}<\cos^2 \theta_*\Bigr)
}
{\sum_{k=1}^{m+1} a_{\gamma,{k}} \int_{\mathcal{X}} w_{k}(x,v)\, \dd x\, \dd v}
\end{equation}
with
\[
 w^+_{k}(x,v) := \frac{1}{\Omega_{k}\Omega_{n-k}}\sum_{d=k-1}^m \1_{\{x\in\partial M_d\}}|\tr_{d-k+1} H_d(x,v)|, \quad 1\le k\le m+1,
\]
and consequently
\[
 (1-\bar\Delta) P_\mathrm{tube}(c) < P(c) < (1+\bar\Delta) P_\mathrm{tube}(c).
\]
The upper bound (\ref{bound0}) simplifies significantly for $m=\dim M=0$ (see Remark \ref{rem:bound}).
\end{remark}
Now consider the case of a light-tailed or subexponential (but not regularly varying) distribution of $R_n$:
\begin{theorem}
\label{thm:MainS*}
Suppose that $F_{R_n}\in\mathcal{L}_{\beta,\gamma}$ with $\beta<1$ or $(\beta, \gamma)=(1,+\infty)$.
Then, as $c\to +\infty$, one has
\begin{equation}
\label{eq:MainS*}
 \Delta(c) \sim
 \Bigl\{c^{2(1-\beta)} \ell_0\bigl(c^2\bigr)\Bigr\}^{\frac{n-m-3}{2}}
 \frac{\sum_{k=1}^{m+1}
 \frac{1}{B\bigl(\frac{k}2,\frac{n-k}2\bigr)}
 \int_{\mathcal{X}} w_{k}(x,v) G_{k}(\theta(x,v),c) \,\dd x\,\dd v}
      {\frac{\Gamma\bigl(\frac{n-m-1}{2}\bigr)}{B\bigl(\frac{m+1}{2},\frac{n-m-1}{2}\bigr)} \frac{\mathrm{Vol}_m(M)}{\Omega_{m+1}}},
\end{equation}
where
\[
 G_k(\theta,c) :=
 \cos^{k-2\beta+2}\theta\sin^{n-k-2}\theta
 \exp\Bigl\{ -c^{2(1-\beta)} \ell_0\left(c^2\right) g_\beta(\cos^2\theta) -r_\beta(c^2,\cos^2\theta)\Bigr\}
\]
and $\mathrm{Vol}_m(M) = \int_{M} \dd x$
is the $m$-dimensional volume of the index manifold $M$.
\end{theorem}

\begin{proof}
We substitute the asymptotics for $D_k(\theta,c)$ from (\ref{eq:D_k_asy}) into (\ref{eq:MainL}).
In the denominator on the right-hand side of (\ref{eq:MainL}), the term with $k=m+1$ dominates as $c\to +\infty$.
Since
$\mathrm{Vol}_{m}(M)=\frac{1}{\Omega_{n-m+1}}\int_{\partial M_m\times S(N_x M)}\dd x\dd v$,
the denominator is asymptotically equivalent to
\begin{align*}
D_{m+1}(0,c)\int_{\mathcal{X}} w_{m+1}(x,v)\,\dd x\,\dd v
= \frac{\Gamma\bigl(\frac{n-m-1}{2}\bigr)}{B\bigl(\frac{m+1}2,\frac{n-m-1}2\bigr) \{c^{2(1-\beta)} \ell_0(c^2)\}^{\frac{n-m-1}{2}}} \frac{\mathrm{Vol}_{m}(M)}{\Omega_{m+1}}.
\end{align*}
\end{proof}
One can see that
the right-hand side of \eqref{eq:MainS*} tends to zero exponentially fast with leading term
\[
 O\left( \exp\left\{ -c^{2(1-\beta)} \ell_0\left(c^2\right) g_\beta(\cos^2\theta_{*}) \right\} \right),
\] 
so the tube formula approximation is tail-valid in this case.

\subsection{Bonferroni approximation}
\label{subsec:Bonferr}

We apply Theorems \ref{thm:MainRV} and \ref{thm:MainS*} to the case of a finite index set $M=\{u_1,\ldots,u_N\}\subset\S^{n-1}$.
The relative error $\Delta(c)$ now simplifies to
\begin{equation}
\label{delta_c-dim0}
 \Delta(c) = \frac{\sum_{1\le i\le N}\Pr(T_i \ge c) -\Pr\Bigl(\max_{1\le i\le N}T_i \ge c\Bigr)}{\sum_{1\le i\le N}\Pr(T_i \ge c)},
\end{equation}
where $T_i=\langle u_i,\xi\rangle$.

For each point $u_i\in M$, the normal cone $N_{u_i}M$ is the linear subspace orthogonal to $u_i$, hence 
\[
 S(N_{u_i}M)=\{ v \mid \langle v,u_i\rangle=0,\ \Vert v\Vert=1 \}.
\]
In the statements below, we replace the integral $\Omega_{n-1}^{-1}\int_{S(N_{u_i}M)}\,\cdot\ \dd v$ by the expectation on $V_i\sim\mathrm{Unif}(S(N_{u_i}M))$.

\begin{theorem}
\label{thm:MainRV_dim0}
For $F_{R_n}\in\mathrm{RV}_{-\gamma}$ with $\gamma\in (0,+\infty)$ one has
\begin{equation}
\label{D-RV}
 \Delta(c) \sim \frac{\sum_{i=1}^N \Pr\Bigl(\widetilde B<\cos^2 \theta(x_i,V_i)\Bigr)}{N}
\end{equation}
as $c\to +\infty$, where $\widetilde B$ is the random variable distributed as the beta distribution $B_{\gamma+\frac{1}{2},\frac{n-1}{2}}$.
Moreover, it holds
\begin{equation}
\label{D-RV_bound}
 0<\lim_{c\to +\infty}\Delta(c)<\bar\Delta := \Pr\Bigl(\widetilde B<\cos^2\theta_*\Bigr),
\end{equation}
that is, the Bonferroni approximation is not tail-valid in this case.
\end{theorem}

\begin{proof}
When $m=0$, the integral $\int_{M} \dd x$ means the summation $\sum_{x\in M}$, and $\Omega_1=2$.
The numerator in (\ref{eq:MainRV}) becomes
\[
 a_{\gamma,1}\sum_{i=1}^N \int_{S(N_{x_i}M)} \Pr\Bigl(\widetilde B_{\gamma,1}<\cos\theta(x_i,v)\Bigr)\,\dd v/(2\Omega_{n-1}) = a_{\gamma,1}\sum_{i=1}^N \Pr\Bigl(\widetilde B_{\gamma,1}<\cos\theta(x_i,V_i)\Bigr)/2.
\]
Since $\mathrm{Vol}_0(M)=\int_{M} \dd x=N$, the denominator in (\ref{eq:MainRV}) becomes $a_{\gamma,1} N/2$.
\end{proof}

\begin{remark}
\label{rem:bound}
 The bounds in (\ref{D-RV_bound}) imply that 
\begin{equation}
\label{bound}
(1-\bar\Delta) N \Pr(T_1 \ge c) < \Pr\biggl(\max_{1\le i\le N}T_i \ge c\biggr) < N \Pr(T_1 \ge c)
\end{equation}
for large $c$.
The performance of this bound will be examined through numerical experiments in Section \ref{sec:num}.
\end{remark}

The proof of the following theorem is provided in Section \ref{sec:proof}.
\begin{theorem}
\label{thm:MainS*_dim0}
For $F_{R_n}\in\mathcal{L}_{\beta,\gamma}$ with $\beta<1$ or $(\beta,\gamma)=(1,+\infty)$,
one has
\begin{equation*}
 \Delta(c) \sim \frac{D}{N}\frac{\cos^{n(1-\beta)}\theta_*}{2\sqrt{\pi}\tan\theta_*}
 \Bigl\{c^{2(1-\beta)}\ell_0\left(c^2\right)\Bigr\}^{-\frac{1}{2}}
 \exp\Bigl\{ -c^{2(1-\beta)} \ell_0\left(c^2\right) g_\beta(\cos^2\theta_*) -r_\beta(c^2,\cos^2\theta_*)\Bigr\}
\end{equation*}
as $c\to +\infty$, where
\begin{equation}
\label{D}
 D = \sharp\bigl\{ (i,j) \mid \rho_{ij} = \rho_{*} \bigr\}.
\end{equation}
Equivalently,
\begin{align*}
 \log\Delta(c)
 =& -c^{2(1-\beta)} \ell_0\bigl(c^2\bigr) g_\beta(\cos^2\theta_*) 
 - \frac{1}{2} \log\Bigl\{ c^{2(1-\beta)}\ell_0\bigl(c^2\bigr) \Bigr\}
 - r_\beta\bigl(c^2,\cos^2\theta_*\bigr) \\
 & +\log\frac{\cos^{n(1-\beta)}\theta_*}{2\sqrt{\pi}\tan\theta_*} +\log\frac{D}{N} +o(1)
\end{align*}
as $c\to +\infty$.
Since $r_\beta\bigl(c^2,\cos^2\theta_*\bigr)=o\bigl(c^{2(1-\beta)}\ell_0(c^2)\bigr)$,
the Bonferroni approximation is tail-valid.
\end{theorem}
Now apply Theorem \ref{thm:MainS*_dim0} to some of the distributions from Example \ref{ex:L}.
\begin{example}
\label{ex:three_dist}

\begin{enumerate}[label={\rm (\roman*)}]

\item
Suppose that ${R_n}$ follows the log-normal distribution $\exp\mathcal{N}(0,1)$ with tail probability function (\ref{log_normal}).
Setting $\beta=1$, $\ell(x)=\log x+1/\log x$, $\ell_0(x)=\log x$ and $g_1(y)=-\log y$, we obtain
\[
 r_1(h,y)=\int_1^{y^{-1}}\left\{ \log(h v)+\frac{1}{\log(h v)} -\log h \right\}v^{-1}\dd v = \frac{1}{2}(\log y)^2 + o(1), \quad h\to +\infty,
\]
and
\begin{equation}
\begin{aligned}
\label{logD-lognormal}
 \log\Delta(c) =& -\log(c^2)(-\log\cos^2\theta_*) - \frac{1}{2}\log \log(c^2) \\
 & - \frac{1}{2}(\log\cos^2\theta_*)^2 + \log\frac{1}{2\sqrt{\pi}\tan\theta_*}
 + \log \frac{D}{N} +o(1), \quad c\to +\infty.
\end{aligned}
\end{equation}

\item
Suppose that ${R_n}\sim \chi^2_n$ with tail probability function (\ref{chi2}).
Setting $\beta=0$, $\ell(x)=(1/2)-(n-2)/(2x)$, $\ell_0(x)=1/2$, $g_0(y)=y^{-1}-1$, we obtain
\[
 r_0(h,y) = h\int_1^{y^{-1}}\left(-\frac{n-2}{2 h v}\right) \dd v = \frac{n-2}{2}\log y
\]
and
\begin{equation}
\label{logD-Gauss}
 \log\Delta(c) = -(1/2)c^{2}\tan^2\theta_* -\frac{1}{2}\log(c^2/2)
 + \log\frac{\cos^{2}\theta_*}{2\sqrt{\pi}\tan\theta_*}
 + \log \frac{D}{N} +o(1)
\end{equation}
as $c\to +\infty$, which was shown in \cite{kuriki-takemura:2001,taylor-takemura-adler:2005}.

\item
Suppose that ${R_n}$ is Bessel-distributed with tail probability function (\ref{Bessel}).
Setting $\beta=1/2$, $\ell(x)=(1/2)-(\nu_1+\nu_2-3)/(4\sqrt{x})$, $\ell_0(x)=1/2$ and
$g_{1/2}(y)=2(y^{-1/2}-1)$, we obtain
\begin{align*}
 r_{1/2}(h,y) = h^{1/2}\int_{1}^{y^{-1}}\left(-\frac{\nu_1+\nu_2-3}{4}\frac{1}{\sqrt{h v}}\right) v^{-1/2}\,\dd v
= \frac{\nu_1+\nu_2-3}{4} \log y
\end{align*}
and
\begin{equation}
\label{logD-Bessel}
 \log\Delta(c) = -c(\sec\theta_*-1)
 - \frac{1}{2} \log (c/2)
 + \log\frac{\cos^{\frac{n-\nu_1-\nu_2+3}{2}}\theta_*}{2\sqrt{\pi}\tan\theta_*}
 + \log\frac{D}{N} +o(1)
\end{equation}
as $c\to +\infty$.
\end{enumerate}
\end{example}

\begin{remark}
Let $(T_1,\ldots,T_N)$ be a centered Gaussian random vector with covariance matrix $\Sigma=(\rho_{ij})$ and unit variances $\rho_{ii}=1$.
The density function of $(T_1,\ldots,T_N)$ is denoted by $p_T(x_1,\ldots,x_N;\Sigma)$.
The multivariate Mills' ratio \cite{savage:1962} states that
\begin{equation}
\label{mills}
 \Pr(\min\{T_1,\ldots,T_N\} \ge c) \sim \frac{p_T(c,\ldots,c;\Sigma)}{c^N \prod_{i}(\Sigma^{-1} \1)_i}
 \asymp c^{-N}\exp\left(-\frac{c^2}{2} \1^\top\Sigma^{-1}\1\right),
\end{equation}
where $\1=(1,\ldots,1)^\top$.
For any partition $\Sigma=\begin{pmatrix} \Sigma_{11} & \Sigma_{12} \\ \Sigma_{21} & \Sigma_{22} \end{pmatrix}$, 
$\Sigma^{-1}-\begin{pmatrix} \Sigma_{11}^{-1} & 0 \\ 0 & 0\end{pmatrix}$ is non-negative definite, and $\1^\top\Sigma^{-1}\1\ge\1^\top\Sigma_{11}^{-1}\1$ holds.
Hence, for any $k=1,\ldots, N$ we have
\[
 \Pr(\min\{T_1,\ldots,T_{k}\} \ge c)/\Pr(\min\{T_1,\ldots,T_{k-1}\} \ge c)=o(1), \quad c\to +\infty.
\]
Therefore, in the right-hand side of the inclusion-exclusion formula
\begin{align*}
& \Pr\Bigl(\max_{i=1,\ldots,N} T_i \ge c\Bigr) - \sum_{i=1}^N \Pr(T_i \ge c) \\
& = -\sum_{i<j} \Pr(\min\{T_i,T_j\} \ge c) +\sum_{i<j<k} \Pr(\min\{T_i,T_j,T_k\} \ge c) - \cdots,
\end{align*}
the first term dominates as $c\to +\infty$,
and the relative error of the tube formula for the maximum $\max\{T_1,\ldots,T_N\}$ is
\[
 \Delta(c) = \frac{N\Pr(T_1 \ge c)-\Pr(\max\{T_1,\ldots,T_N\} \ge c)}{N\Pr(T_1 \ge c)} \sim \frac{\sum_{1\le i<j\le N}\Pr(\min\{T_i,T_j\} \ge c)}{N\Pr(T_1 \ge c)}.
\]
By applying the Mills's formula (\ref{mills}), we have
\[
\Pr(\min\{T_i,T_j\} \ge c)\sim \frac{\frac{1}{2\pi\sqrt{1-\rho_{ij}^2}}\,e^{-\frac{c^2}{1+\rho_{ij}}}}{c^2\bigl(1+\rho_{ij}\bigr)^{-2}} \quad\mbox{and}\quad
 \Pr(T_1 \ge c) \sim \frac{1}{c\sqrt{2\pi}}e^{-\frac{1}{2}c^2},
\]
as $c\to +\infty$, and hence
\[
  \Delta(c) \sim \frac{D}{N} \frac{1}{\sqrt{2\pi}}\left(\frac{1+\rho_*}{2}\right)\left(\frac{1-\rho_*}{1+\rho_*}\right)^{-\frac{1}{2}} c^{-1}\exp\left(-\frac{c^2}{2}\frac{1-\rho_*}{1+\rho_*}\right),
\]
where $\rho_*=\max_{i\neq j}\rho_{ij}$ and $D$ is defined in (\ref{D}).
This is equivalent to (\ref{logD-Gauss}).
\end{remark}

\begin{remark}
\label{rem:copula}
For $N=2$, the relative error $\Delta(c)$ in (\ref{delta_c-dim0}) can be interpreted as a measure of upper tail dependence in the random vector $(T_1,T_2)$ as follows.
%
Let $F$ be the common cumulative distribution function of $T_i$ and $F^{-1}$ its quantile function, $i=1,2$.

%
In extreme value theory, the upper tail dependence coefficient is defined as $\lambda_U:=\lim_{u\uparrow 1} \lambda(u)$, where
\[
 \lambda(u) := 
 \Pr\bigl( T_1 \ge F^{-1}(u)\,|\,T_2 \ge F^{-1}(u)\bigr), \quad u\in[0,1],
\]
see \cite[Definition 2.3]{hult-lindskog:2002}.
The random variables $T_i$, $i=1,\,2$ are said to be upper tail-independent if $\lambda_U=0$, 
and upper tail-dependent, if $\lambda_U>0$.

In our context, setting $c=F^{-1}(u)\to +\infty$ as $u\uparrow 1$, we write
\begin{align*}
\lim_{c\to +\infty }\Delta(c)
= \lim_{u\uparrow 1 }\Delta(F^{-1}(u))
= \lim_{u\uparrow 1 } \frac{\Pr(\min\{ T_1,T_2 \} \ge F^{-1}(u))}{2 \Pr(T_1 \ge F^{-1}(u))}
= \lim_{u\uparrow 1 } \frac{\lambda(u)}{2} = \frac{\lambda_U}{2}.
\end{align*}
Hence, tail-validity $\lim_{c\to +\infty}\Delta(c)=0$ of the Bonferroni approximation with $N=2$ is equivalent to upper tail-independence of $T_1$ and $T_2$. 

\end{remark}

\section{Numerical analysis for finite index sets $M$}
\label{sec:num}

We conduct numerical studies to examine whether the asymptotic results for the relative approximation error $\Delta(c)$ as $c\to +\infty$ from Section \ref{subsec:Bonferr} perform well also for moderate values of $c$.
We set the number $N$ of test statistics $T_i$ to be $N=3$, where $T_i=\langle u_i,\xi\rangle$ with $\xi\in\R^3$ (i.e., $n=3$) being a spherically contoured random vector and $u_i\in\S^2$, $i=1,2,3$.
Thus, $(T_i)_{i=1,2,3}$ is a finite spherically contoured random field.
Its correlation structure $\rho_{ij}=\langle u_i,u_j\rangle$, $i,j=1,2,3$ is set to be
$\rho_{ii}=1$ and $\rho_{ij}=1/4$, $i\ne j$.
Then, the critical radius in (\ref{theta*}) is
\[
 \theta_* = \cos^{-1}\sqrt{5/8}.
\]
The multiplicity $D$ defined in (\ref{D}) is equal to $6$.

For the distribution of
$R_3 = \Vert\xi\Vert^2 = \sum_{i,j=1}^3 T_i T_j \rho^{ij}$,
where $(\rho^{ij})=(\rho_{ij})^{-1}$,
we consider four probability laws  from Example \ref{ex:L}.
Figures \ref{fig:t}, \ref{fig:lognormal}, \ref{fig:Bessel}, and \ref{fig:gauss} show the logarithm of the tail probability $\log\Pr(T_{\max}>c)$ and the logarithm of its relative error $\Delta(c)$.

In Figure \ref{fig:t}, we set $(T_i)_{i=1,2,3}$ to be proportional to the trivariate $t$-random vector with $\nu=3$ degrees of freedom.
This is implemented by letting $s$ in (\ref{Xi_u-again}) be distributed as $s^2\sim \chi^2_3$ or $R_3=\Vert\xi\Vert^2\sim F_{3,3}$.
Note that $\E[R_3]=3$.
The left panel shows $\log\Pr(T_{\max}>c)$ by simulation (in blue) and exact calculation (in green), as well as its Bonferroni approximation (in red).
Even when the threshold $c$ increases, the difference (i.e., asymptotic bias) between the simulated/exact values and the Bonferroni approximation does not disappear,
indicating that the Bonferroni approximation is not tail-valid.
However, this bias is well captured by the lower bound (\ref{bound}) (in gray) even if $c$ is not necessarily large.
\begin{figure}[ht]
\begin{center}
\scalebox{0.65}{\includegraphics{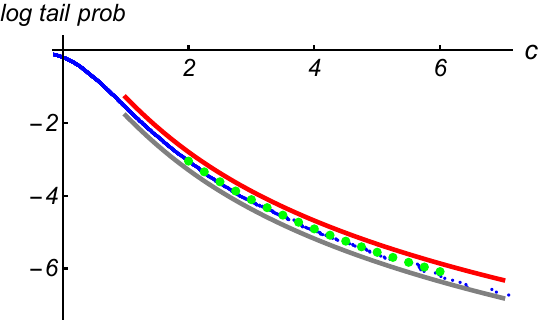}}\qquad\qquad
\scalebox{0.65}{\includegraphics{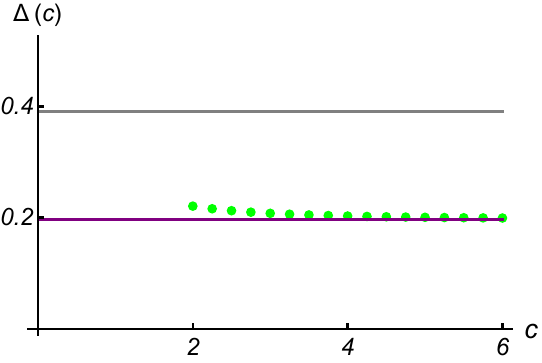}}
\caption{Values of $\log\Pr(T_{\max}>c)$ (left) and $\Delta(c)$ (right) for the $t_3$-distributed random field $(T_i)_{i=1,2,3}$.}
\label{fig:t}
\smallskip
\small
\begin{tabular}{l@{\ }l}
Left: & Simulation with $10000$ trials {\color{blue}(blue)}, Bonferroni approximation {\color{red}(red)}, exact {\color{green} (green)}, \\ 
& lower bound in (\ref{bound}) (gray) \\
Right:& Asymptotic bias $\lim_{c\to +\infty}\Delta(c)$ in (\ref{D-RV}) {\color{purple}(purple)}, upper bound $\bar\Delta$ in (\ref{D-RV_bound}) {\color{black}(gray)}, \\
      & exact {\color{green}(green)}
\end{tabular}
\end{center}
\end{figure}
The right panel of Figure \ref{fig:t} displays the relative error $\Delta(c)$.
When $c\to +\infty$, $\Delta(c)$ converges to a positive constant (in purple) thus validating the result of Theorem \ref{thm:MainRV_dim0}.
Also, $\Delta(c)$ is bounded from above by a constant $\bar\Delta$ (in gray).

In Figure \ref{fig:lognormal}, we set the distribution of $R_3$ to be the multiple $3 e^{-1/2} \mathrm{exp}\,\mathcal{N}(0,1)$ of the Log-normal distribution. Note that $\E[R_3]=3$.
Differently from Figure \ref{fig:t} and well in accordance with Theorem \ref{thm:MainS*_dim0}, the asymptotic bias of the Bonferroni approximation is not observed in the left panel.

\begin{figure}[ht]
\begin{center}
\scalebox{0.65}{\includegraphics{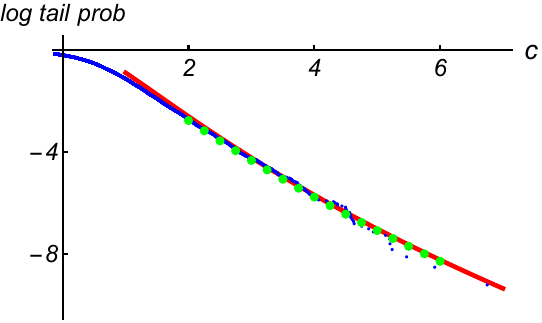}}\qquad\qquad
\scalebox{0.65}{\includegraphics{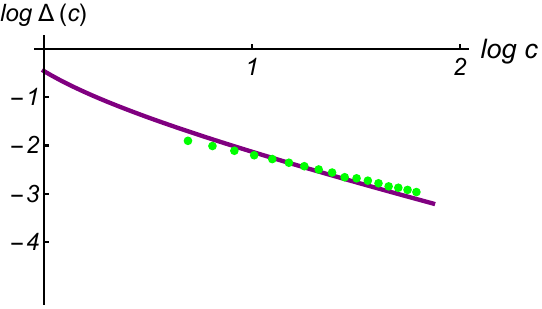}}
\caption{Values of $\log\Pr(T_{\max}>c)$ (left) and $\Delta(c)$ (right) for the lognormal-distributed random field $(T_i)_{i=1,2,3}$.
}
\label{fig:lognormal}
\smallskip
\small
\begin{tabular}{l@{\ }l}
Left: & Simulation with $10000$ trials {\color{blue}(blue)}, Bonferroni approximation {\color{red}(red)}, exact {\color{green} (green)} \\ 
Right:& Asymptotics in (\ref{logD-lognormal}) {\color{purple}(purple)}, exact {\color{green}(green)} 
\end{tabular}
\end{center}
\end{figure}
The right panel of Figure \ref{fig:lognormal} displays the graph of $(\log c,\log\Delta(c))$.
The vertical values $\log\Delta(c)$ are computed using the formula (\ref{logD-lognormal}) with $c:=c\times (\sqrt{e}/3)^{\frac{1}{2}}$.
The relationship between $\log c$ and $\log\Delta(c)$ is asymptotically linear with a slope of $(-2)\times(-\log\cos^2\theta_*)=2\log(5/8)\approx -0.94$ when $c\to +\infty$.
The asymptotic formula (\ref{logD-lognormal}) for $\Delta(c)$ (in purple) approximates the exact values (in green) well.
The Bonferroni approximation in this setting is tail-valid.

In Figure \ref{fig:Bessel}, we set  
$R_3=\Vert\xi\Vert^2\sim 4^{-1}\chi_3^2\chi_4^2$ to be a multiple of the Bessel distribution with degrees of freedom $3$ and $4$. 
This corresponds to the case where $s$ in (\ref{Xi_u-again}) has the distribution $s^{-2}\sim\chi^2_4/4$ and ensures $\E[R_3]=3$.
The values
$\log\Delta(c)$ are computed by formula (\ref{logD-Bessel}) with $c:=c\times \sqrt{4}$.

\begin{figure}[ht]
\begin{center}
\scalebox{0.65}{\includegraphics{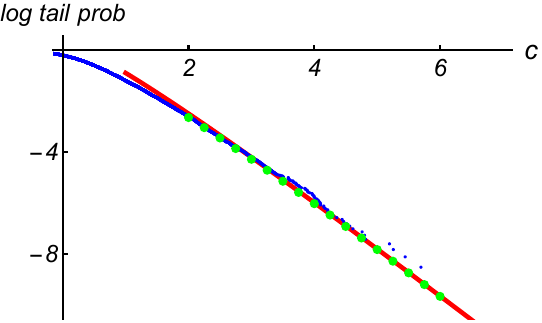}}\qquad\qquad
\scalebox{0.65}{\includegraphics{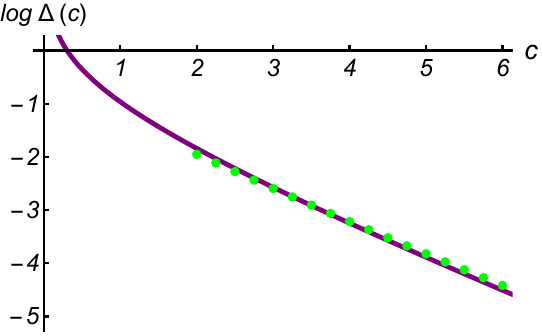}}
\caption{Values of $\log\Pr(T_{\max}>c)$ (left) and $\Delta(c)$ (right) for the Bessel-distributed random field $(T_i)_{i=1,2,3}$.}
\label{fig:Bessel}
\smallskip
\small
\begin{tabular}{l@{\ }l}
Left: & Simulation with $10000$ trials {\color{blue}(blue)}, Bonferroni approximation {\color{red}(red)}, exact {\color{green} (green)} \\ 
Right:& Asymptotics in (\ref{logD-Bessel}) {\color{purple}(purple)}, exact {\color{green}(green)} 
\end{tabular}
\end{center}
\end{figure}

As in Figure \ref{fig:lognormal}, the Bonferroni method provides a good approximation to the simulated/exact values of tail probabilities of the maximum, and the error value $\Delta(c)$ in (\ref{logD-Bessel}) is asymptotically very close to the exact value (in green).
The right panel of Figure \ref{fig:Bessel} shows the graph of $(c,\log\Delta(c))$.
In this scaling, it exhibits asymptotically linear behavior with a slope of $-(\sec\theta_*-1)\times\sqrt{4}=-2\sqrt{8/5}+2\approx -0.53$ as $c\to +\infty$.

Figure \ref{fig:gauss} showcases the trivariate centered Gaussian random vector $(T_i)_{i=1,2,3}$ with $R_3\sim\chi_3^2$.
This corresponds to the case $s\equiv 1$ in (\ref{Xi_u-again}).
Obviously, $\E[R_3]=3$.
Due to tail validity of the Bonferroni approximation, its features closely resemble those displayed in Figures \ref{fig:lognormal} and \ref{fig:Bessel}.
In the right panel, the graph of $(c^2,\log\Delta(c))$ appears almost linear asymptotically with a slope of $-(1/2)\tan^2\theta_*=-3/10$ as $c\to +\infty$.

\begin{figure}[ht]
\begin{center}
\scalebox{0.65}{\includegraphics{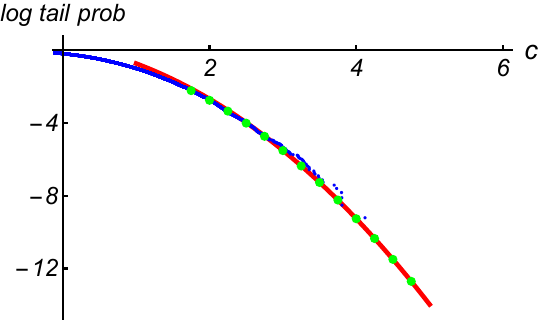}}\qquad\qquad
\scalebox{0.65}{\includegraphics{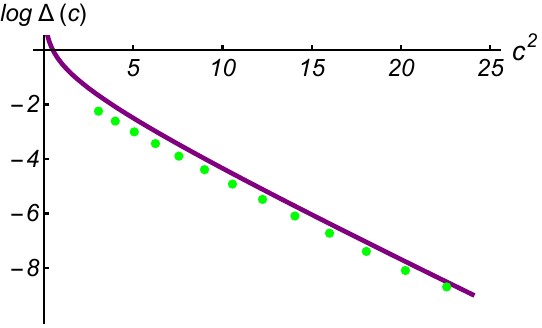}}
\caption{Values of $\log\Pr(T_{\max}>c)$ (left) and $\Delta(c)$ (right) for the Gaussian random field $(T_i)_{i=1,2,3}$.}
\label{fig:gauss}
\smallskip
\small
\begin{tabular}{l@{\ }l}
Left: & Simulation with $10000$ trials {\color{blue}(blue)}, Bonferroni approximation {\color{red}(red)}, exact {\color{green} (green)} \\ 
Right:& Asymptotics in (\ref{logD-Gauss}) {\color{purple}(purple)}, exact {\color{green}(green)} 
\end{tabular}
\end{center}
\end{figure}

\section{Proofs}
\label{sec:proof}

In this section, we prove Lemmas \ref{lem:MainL} and \ref{lem:MainD} as well as Theorem 
\ref{thm:MainS*_dim0} from Section \ref{sec:Main}.

\begin{proof}[Proof of Lemma \ref{lem:MainL}]
By the property of the elliptically contoured laws \cite[Lemma 2]{cambanis-etal:1981}, \cite[Lemmas 2.5, 2.6]{schmidt:2002}, $R_k$ and $\bar R_{n-k}$ have the joint distribution
\[
 (R_k,\bar R_{n-k}) \mathop{=}^d (R_n B_k, R_n (1-B_k)),
\]
where $B_k\sim B_{\frac{k}{2},\frac{n-k}{2}}$ is a random variable independent of $R_n$.
The probabilities in expressions (\ref{tube}) and (\ref{max}) can be written as
\begin{align*}
 \Pr\bigl(R_k\ge c^2, \bar R_{n-k}\le t(x,v) R_k\bigr)
 =& \Pr\bigl(R_n B_k\ge c^2, 1-B_k\le t(x,v) B_k\bigr) \\
 =& \int_{1/(1+t(x,v))}^1 \Pr\bigl(R_n\ge c^2/B_k\,|\, B_k\bigr)\, \dd P_{B_k},
\end{align*}
and
\[
 \Pr\bigl(R_k\ge c^2\bigr)
 = \int_0^1 \Pr\bigl(R_n\ge c^2/B_k\,|\, B_k\bigr)\, \dd P_{B_k},
\]
respectively,
where $\dd P_{B_k}$ is the probability measure of $B_{\frac{k}{2},\frac{n-k}{2}}$.
Then we can express
\begin{equation}
\label{diff}
 P_\mathrm{tube}(c)-P(c) = \sum_{k=1}^{m+1} \int_{\mathcal{X}} w_{k}(x,v) \int_{0}^{\cos^2\theta(x,v)} \Pr\bigl(R_n\ge c^2/B_{k}\,|\, B_{k}\bigr)\, \dd P_{B_{k}}\, \dd x\, \dd v
\end{equation}
and, analogously,
\begin{equation}
\label{Ptube}
 P_\mathrm{tube}(c) = \sum_{k=1}^{m+1} \int_{\mathcal{X}} w_{k}(x,v) 
 \, \dd x\, \dd v \int_{0}^{1} \Pr\bigl(R_n\ge c^2/B_{k}\,|\, B_{k}\bigr)\, \dd P_{B_{k}}.
\end{equation}
Divide both (\ref{diff}) and (\ref{Ptube}) by $\Pr\bigl(R_n\ge c^2\bigr)$ and let $c$ tend to infinity.
Using representation \eqref{eq:FL}, we obtain
\begin{align*}
 I_{n,k}(c,y)
 &:= \frac{ \Pr\bigl(R_n\ge c^2/B_k\,|\, B_{k}=y\bigr)}{\Pr\bigl(R_n\ge c^2\bigr)}
 = \frac{ \Pr\bigl(R_n\ge c^2/y\bigr)}{\Pr\bigl(R_n\ge c^2\bigr)} \\
 &= \frac{C\bigl(c^2/y\bigr)}{C\bigl(c^2\bigr)}\exp\left\{-\int_{c^2}^{c^2/y} q(t)\, \dd t \right\}
\sim \exp\left\{-\int_{c^2}^{c^2/y} q(t)\, \dd t \right\}
\end{align*}
as $c\to +\infty$.
Since the right-hand side is bounded by 1, we may use the Lebesgue theorem on dominated convergence to deduce that
\[
 \int_{0}^{\cos^2 \theta} I_{n,k}\bigl(c, B_{k}\bigr)\, \dd P_{B_{k}} \sim \int_{0}^{\cos^2 \theta} \exp\left\{-\int_{c^2}^{c^2/B_{k}} q(t)\, \dd t \right\}\, \dd P_{B_{k}}=D_k(\theta,c),
\]
\[
 \int_{0}^{1} I_{n,k}\bigl(c, B_{k}\bigr)\, \dd P_{B_{k}} \sim D_k(0,c),\quad c\to +\infty.
\]
A reapplication of the Lebesgue theorem justifies interchanging the integral over $\mathcal{X}$ and the limit as $c\to +\infty$ due to the absolute integrability of the functions $w_{k}$.
The following elementary fact which can be proven by induction completes the proof:
Let functions $A_j, B_j:\R_+\to \R_+$, $j=1,\ldots, l$ be such that $A_j(c)\sim B_j(c)$,
$B_i(c)/B_j(c)\to p_{ij}\in \bar\R$ as $c\to +\infty$ for all $i,j=1,\ldots,l$.
Then $\sum_{j=1}^l A_j(c)\sim \sum_{j=1}^l B_j(c)$ as $c\to +\infty$.
\end{proof}

\begin{proof}[Proof of Lemma \ref{lem:MainD}] 

By Lemma \ref{lem:MainL}, we have to find the asymptotics of $D_k(\theta,c)$ in \eqref{eq:D_k} as $c\to +\infty$.
We prove the two cases
\begin{itemize}
\setlength{\parskip}{0mm}
\setlength{\itemsep}{1mm}
\item[(i)]
$\beta<1$ or $(\beta,\gamma)=(1,+\infty)$; (i-a) $\theta>0$, (i-b) $\theta=0$,

\item[(ii)]
$\beta=1$ and $\gamma<+\infty$,
\end{itemize}
separately.

Case (i).
Using the notation $a:=\cos^2 \theta \in [0,1]$, $h=c^2$, $p=k/2$, $q=(n-k)/2$, write $D_k(\theta,c)$ in (\ref{eq:D_k}) as
\begin{equation}
\label{eq:D_k_mod}
  B(p,q)^{-1} \int_0^a e^{-J} y^{p-1} (1-y)^{q-1}\, \dd y,
\end{equation}
where $J:=\int_{h}^{h/y} \ell(t) t^{-\beta}\, \dd t$.
After the substitution of variables $t=hv$, we may write due to the slow variation of $\ell$ that
\begin{equation}
\label{eq:J}
 J = h^{1-\beta} \int_1^{y^{-1}} \frac{\ell(hv)}{\ell(h)} \ell(h) v^{-\beta}\, \dd v \sim h^{1-\beta} \ell_0(h) g_\beta(y), \quad h\to +\infty,
\end{equation}
where $\ell_0(h)$ is an arbitrary function such that $\ell_0(h)\sim\ell(h)$ as $h\to +\infty$, and the function $g_\beta$ is defined in \eqref{eq:g}.
 Introduce the notation $b:=h^{1-\beta} \ell_0(h)$.
It holds $b \to +\infty$ as $h\to +\infty$.
In our analysis, we need to provide an asymptotic expansion for $J$ up to first order in $b$.
Noting (\ref{eq:J}), we see that
\[
  J - h^{1-\beta} \ell_0(h) g_\beta(y) = h^{1-\beta} \int_1^{y^{-1}} \bigl\{\ell(hv) - \ell_0(h)\bigr\} v^{-\beta}\, \dd v=r_\beta(h,y),
\]
where the function $r_\beta(h,y) $ was introduced in \eqref{eq:rbeta_def}.
We substitute $z=g_\beta(y)$ in the integral \eqref{eq:D_k_mod} after plugging the expression \eqref{eq:J} into it.
Note that $g_\beta$ is a monotonically decreasing function with $g_\beta(0)=+\infty$ and $g_\beta(1)=0$.
Then, we have
\begin{align}
 D_k(\theta,c)
 &\sim B(p,q)^{-1} \int_0^{a} e^{-b g_\beta(y) -r_\beta(h,y)} y^{p-1} (1-y)^{q-1} \, \dd y \nonumber \\
 &= B(p,q)^{-1} \int_{g_\beta(a)}^{+\infty} e^{-b z -r_\beta(h,g^{-1}_\beta(z))} \bigl(g^{-1}_\beta(z)\bigr)^{p-1} \bigl(1-g^{-1}_\beta(z)\bigr)^{q-1} \frac{1}{|g'_\beta(g^{-1}_\beta(z))|}\, \dd z \nonumber \\
 &= B(p,q)^{-1} \int_{g_\beta(a)}^{+\infty} e^{-b z}
 e^{-r_\beta(h,g^{-1}_\beta(z))}
 \bigl(g^{-1}_\beta(z)\bigr)^{p-\beta+1} \bigl(1-g^{-1}_\beta(z)\bigr)^{q-1} \,\dd z.
\label{D_k-tmp}
\end{align}

(i-a) When $g_\beta(a)>0$ (which equivalently means $a<1$ and $\theta>0$), expression (\ref{D_k-tmp}) is asymptotically evaluated integrating by parts as
\[
 D_k(\theta,c)
 \sim B(p,q)^{-1} b^{-1} e^{-b g_\beta(a)} 
 e^{-r_\beta(h,a)}
 a^{p-\beta+1} (1-a)^{q-1}, \quad b\to +\infty.
\]

(i-b) Suppose that $g_\beta(a)=0$ (i.e., $a=1$ or $\theta=0$).
Using the Taylor expansion $g^{-1}_\beta(z) = 1-z +o(z)$ around $z=0$ of order one, we have
\begin{align*}
 e^{-r_\beta(h,g^{-1}_\beta(z))} \bigl(g^{-1}_\beta(z)\bigr)^{p-\beta+1} \bigl(1-g^{-1}_\beta(z)\bigr)^{q-1} \sim z^{q-1}.
\end{align*}
Noting this and $r_\beta(h,1)=0$, we have
\begin{align*}
 D_k(0,c)
 &\sim B(p,q)^{-1} \int_{0}^{+\infty} e^{-b z} z^{q-1}\, \dd z \\
 &=B(p,q)^{-1} b^{-q} \int_{0}^{+\infty} e^{-z} z^{q-1}\, \dd z \\
 &= \frac{\Gamma(q)}{B(p,q) b^{q}} = \frac{\Gamma(p+q)}{\Gamma(p) b^{q}}, \quad b\to +\infty.
\end{align*}

Case (ii).
Recall that $ F_{R_n}\in\mathrm{RV}_{-\gamma}$ here.
Let $q(t)$ be the function from representation \eqref{eq:FL} for $\bar F_{R_n}$, which holds true
since $\mathrm{RV}_{-\gamma}\subset \mathcal{L}$.
By definition of regular variation, it holds 
\[
 \exp\left\{ -\int_{y}^{\lambda y} q(t)\, \dd t \right\} \sim \frac{1}{\lambda^\gamma}, \quad y\to +\infty
\]
for any $\lambda>0$.
Thus, we obtain
\begin{align*}
 D_k(\theta,c) & \sim \int_{0}^{\cos^2\theta} B_{k}^\gamma\, \dd P_{B_{k}} \\
 &= \frac{a_{\gamma,k}}{B\bigl(\gamma+\frac{k}{2},\frac{n-k}{2} \bigr)} \int_{0}^{\cos^2\theta} y^{\gamma+\frac{k}{2}-1} (1-y)^{\frac{n-k}{2}-1}\, \dd y \\
 &= a_{\gamma,k} \Pr\Bigl(\widetilde B_{\gamma,k}<\cos^2 \theta\Bigr), \quad c\to +\infty,
\end{align*}
where $\widetilde B_{\gamma,k}\sim B_{\gamma+\frac{k}{2},\frac{n-k}{2}}$ and $D_k(0,c)= a_{\gamma,k}.$
\end{proof}

\begin{proof}[Proof of Theorem \ref{thm:MainS*_dim0}]

Let $b=c^{2(1-\beta)} \ell_0\bigl(c^2\bigr)$ as before.
Evidently, assumption ``$\beta<1$ or $\beta=1$, $\gamma=\lim_{h\to +\infty}\ell_0(h)=+\infty$'' of Theorem \ref{thm:MainS*_dim0} is equivalent to  $b\to +\infty$ as $c\to +\infty$.
When the dimension $m=0$, the right-hand side of (\ref{eq:MainS*}) rewrites as
\begin{equation}
\label{Delta_c_tmp}
\begin{aligned}
 \Delta(c)
 &\sim
 b^{\frac{n-3}{2}} \frac{\sum_{i=1}^{N}
 \frac{1}{\Omega_{n-1}}
 \int_{S(N_{u_i}M)} G_{1}(\theta(u_i,v),c) \,\dd v}{N \,\Gamma(\frac{n-1}{2})}, \quad c\to +\infty.
\end{aligned}
\end{equation}

We will evaluate
\begin{equation}
\label{int_dv}
\begin{aligned} 
 \int_{S(N_{u_i}M)} & G_{1}(\theta(u_i,v),c) \,\dd v \\
& \sim\int_{S(N_{u_i}M)}
 \cos^{3-2\beta}\theta\sin^{n-3}\theta
 \exp\bigl\{ -b g_\beta(\cos^2\theta) -r_\beta(c^2,\cos^2\theta) \bigr\} \,\dd v
\end{aligned}
\end{equation}
as $b\to +\infty$.
Without loss of generality, we set $i=1$ and assume that $\rho_*=\max_{j>1}\rho_{1j}$ is attained when $j=2$: 
$\rho_*=\rho_{12}$.
Let $v_0=(u_2-\rho u_1)/\sqrt{1-\rho_*^2}$, and $v_1,\ldots,v_{n-2}$ be unit vectors such that $u_1,v_0,v_1,\ldots,v_{n-2}$ is an orthogonal basis of $\R^n$.
Then, any element $v$ of $S(N_{u_1}M)=\{v\mid \langle v,u_1\rangle=0,\,\Vert v\Vert=1\}$ writes
\begin{equation}
\label{eq:v}
 v = v(v_0;\phi,h) = \cos\phi v_0 + \sin\phi \sum_{i=1}^{n-2} v_i h_i, \ \ %
 \phi\in [0,\pi], \ \ h=(h_1,\ldots,h_{n-2})\in\S^{n-3}.
\end{equation}
The volume element of $S(N_{u_1}M)$ at $v$ is $\dd v = \sin^{n-3}\phi\,\dd\phi\,\dd h$,
where $\dd h$ is the volume element of $\S^{n-3}$.
Recall that
\begin{equation}
\label{eq:cos}
 \cos^2\theta(u_1,v) = \max_{j>1}\frac{\langle u_j,v\rangle^2}{(1-\langle u_1,u_j\rangle)^2+\langle u_j,v\rangle^2}, \quad v\in S(N_{u_1}M),
\end{equation}
which takes its maximum at $v=v_0$.
To see that, recall that $v_0$ is the tangent vector on the geodesic connecting $u_1$ and its nearest neighbor $u_2$
(Remark \ref{rem:voronoi} and Figure \ref{fig:voronoi}, right).
When 
$$v\in B(v_0,\varepsilon):=\{v(v_0;\phi,h) \mid 0\le\phi<\varepsilon,\,h\in\S^{n-3}\}, $$
the maximal value of \eqref{eq:cos} is attained at $j=2$.
Inserting the expression for $v$ from \eqref{eq:v} into \eqref{eq:cos} and using the definition of $v_0$ leads to
\begin{align*}
 \cos^2\theta(u_1,v)
 = \frac{\langle u_2,v\rangle^2}{(1-\langle u_1,u_2\rangle)^2 + \langle u_2,v\rangle^2}
 &= \frac{(1-\rho_*^2) \cos^2\phi }{(1-\rho_*)^2 + (1-\rho_*^2) \cos^2\phi} \\
 &= \frac{(1+\rho_*) \cos^2\phi }{1-\rho_* + (1+\rho_*) \cos^2\phi },
\end{align*}
hence
\begin{align}
 g_\beta(\cos^2\theta(u_1,v))
 &=
 g_\beta\left(\frac{1+\rho_*}{2}\right)
 - \left(1-\frac{1+\rho_*}{2}\right) \left(\frac{1+\rho_*}{2}\right)
 g_\beta'\left(\frac{1+\rho_*}{2}\right) \phi^2
 + O(\phi^3) \nonumber \\
 &= g_\beta(\cos^2\theta_*) + K \phi^2 + O(\phi^3),
\label{g_expansion}
\end{align}
where $K=-(\sin^2\theta_* \cos^2\theta_*) g'_\beta(\cos^2\theta_*) = \sin^2\theta_* \cos^{-2(1-\beta)}\theta_*$.

On the other hand, it holds by (\ref{eq:rbeta_def-o1}) that
\begin{align*}
 r_\beta(c^2,\cos^2(u_1,v)) - r_\beta(c^2,\cos^2\theta_*) 
 &= \frac{\partial}{\partial y}r_\beta(c^2,\cos^2\bar\theta)(\cos^2\theta(u_1,v)-\cos^2\theta_*) \\ 
 &= o(b) g'_\beta(\cos^2\bar\theta)(\cos^2\theta(u_1,v)-\cos^2\theta_*),
\end{align*}
where $o(b)$ is a factor such that $o(b)/b\to 0$ as $b\to +\infty$ uniformly in $v$,
and $\bar\theta$ is a point between $\theta(u_1,v)$ and $\theta_*$.
This, combined with
\begin{align*}
 g'_\beta(\cos^2\bar\theta)(\cos^2\theta(u_1,v)-\cos^2\theta_*)
 =& \bigl(-\cos^{-2(2-\beta)}\theta_*+o(1)\bigr) \bigl(-\cos^2\theta_*\sin^2\theta_* \phi^2+O(\phi^3)\bigr) \\
 =& K\phi^2+o(\phi^2),
\end{align*}
means that for arbitrary $\varepsilon_1>0$ and sufficiently large $b$ we have
\begin{equation}
 \bigl|r_\beta(c^2,\cos^2(u_1,v)) - r_\beta(c^2,\cos^2\theta_*)\bigr|<\varepsilon_1 b (K\phi^2 + o(\phi^2)) \quad\mbox{uniformly in $v$.} 
\label{r_bound}
\end{equation}
Now we evaluate the contribution of $v\in B(v_0,\varepsilon)$ in the integral in (\ref{int_dv}) by the Laplace method.
Recall that $b=c^{2(1-\beta)} \ell_0\bigl(c^2\bigr)$.
Let $f(\theta)=\cos^{3-2\beta}\theta\sin^{n-3}\theta$.
Using (\ref{g_expansion}) and (\ref{r_bound}), we have
\begin{align*}
\int_{B(v_0,\varepsilon)}
 & f(\theta(u_1,v)) \exp\bigl\{ -b g_\beta(\cos^2\theta(u_1,v)
 -r_\beta(c^2,\cos^2\theta(u_1,v) \bigr\} \dd v \\
 &\sim 
 f(\theta_*) \exp\bigl\{ -b g_\beta(\cos^2\theta_*) -r_\beta(c^2,\cos^2\theta_*) \bigr\} \int_0^\varepsilon \exp\{ -b K \phi^2 \} \sin^{n-3}\phi\,\dd \phi \int_{\S^{n-3}} \dd h \\
 &\sim 
 f(\theta_*)\exp\bigl\{ -b g_\beta(\cos^2\theta_*) -r_\beta(c^2,\cos^2\theta_*) \bigr\} (b K)^{-\frac{n-2}{2}} 2^{-1}\Gamma\left(\frac{n-2}{2}\right) \Omega_{n-2} \\
 &=
 f(\theta_*)\exp\bigl\{ -b g_\beta(\cos^2\theta_*) - r_\beta(c^2,\cos^2\theta_*)\bigr\} K^{-\frac{n-2}{2}} b^{-\frac{n-2}{2}} \pi^{\frac{n-2}{2}}.
\end{align*}

The contribution of $v\notin B(v_0,\varepsilon)$ to the integral in (\ref{int_dv}) is negligible.
Substituting this into the right-hand side of (\ref{Delta_c_tmp}), we have
\begin{align*}
 \Delta(c)
 \sim&
 \cos^{3-2\beta}\theta_*\sin^{n-3}\theta_*
 \exp\bigl\{-b g_\beta(\cos^2\theta_*) -r_\beta(c^2,\cos^2\theta_*)\bigr\}
 K^{-\frac{n-2}{2}}
 b^{-\frac{n-2}{2}} \pi^{\frac{n-2}{2}} \\
 & \times \frac{D}{N} b^{\frac{n-3}{2}} 2^{-1} \pi^{-\frac{n-1}{2}} \\
 =& \frac{D}{N} 
 \frac{\cos^{n(1-\beta)}\theta_*}{\tan\theta_*}
 \frac{1}{2\sqrt{\pi}} \bigl\{c^{2(1-\beta)}\ell(c^2)\bigr\}^{-\frac{1}{2}}
 \exp\bigl\{ -c^{2(1-\beta)}\ell\bigl(c^2\bigr) g_\beta(\cos^2\theta_*) -r_\beta(c^2,\cos^2\theta_*) \bigr\}.
\end{align*}
\end{proof}

\section{Summary and discussion}
\label{sec:summary}

In this paper, we examined the accuracy of the tube method (or equivalently, the Euler characteristic heuristic) for approximating the excursion probability
\[
 P(c)=\Pr\biggl(\sup_{u\in M}\Xi_u \ge c\biggr).
\]
Here, $\Xi$ is a non-Gaussian isotropic (spherically contoured) random field on $\S^{n-1}$ and 
$M\subset\S^{n-1}$ is a stratified $C^2$-manifold with positive reach.

We introduced a wide class $\mathcal{L}_{*,*}$ of distributions to which the probability law of the squared radial part $\|\xi\|^2$ of the field $\Xi$ belongs.
This class $\mathcal{L}_{*,*}$ includes light-tailed distributions and subexponential (and in particular, regularly varying) ones, which are known to be heavy-tailed (see Definition \ref{def:classL_ab} and Proposition \ref{prop:classL_ab}).

The tube formula is tail-valid in the sense that the relative approximation error $\Delta(c)$ tends to zero as $c$ tends to infinity for light-tailed as well as subexponential (but not regularly varying) distributions of $\|\xi\|^2$ 
 (Theorems \ref{thm:MainS*} and \ref{thm:MainS*_dim0}).
On the other hand, $\Delta(c)$ does not tend to zero for regularly varying distributions, as shown in Theorems \ref{thm:MainRV} and \ref{thm:MainRV_dim0}.
In both cases, the behavior of $\Delta(c)$ is governed by the critical radius $\theta_*$ of the index manifold $M$, as is well known for the Gaussian random fields.

In statistics, the tube method is used for simultaneous inference, including multiple testing or construction of simultaneous confidence bands in nonlinear regression.
The method is known to be highly accurate when the distribution of $\xi$ is multivariate Gaussian.
However, difficulties arise when the unknown variance $\sigma^2$ of the marginal Gaussian distribution of $\xi$ is estimated by an independent plug-in estimator $\widehat\sigma^2$.
In this typical scenario, $\widehat\sigma^2$ is a multiple of a chi-square random variable, making the tail of $\|\xi\|^2$ regularly varying.
We have shown that the tube method (including the Bonferroni approximation) is not tail-valid and thus does not provide an accurate approximation of the tail probability.
To address this issue, we provide a simple bound for the relative approximation error $\Delta(c)$ of the tube formula stated in Remarks \ref{rem:bound0} and \ref{rem:bound}.

The spherically contoured random field $\Xi$ in (\ref{Xi_u-again}) with $s\sim \chi_\nu/\sqrt{\nu}$ is an isotropic $t$-random field on the sphere,
whose marginal distributions are Student-$t_\nu$.
Our results show that the Euler characteristic heuristic  is not tail-valid for this $t$-random field.

Worsley \cite{worsley:1994} defined an isotropic $t$-random field in the Euclidean space as the ratio of an isotropic Gaussian random field to an independent isotropic $\chi$-random field, and derived the expected Euler characteristic of its excursion sets.
Then it has been established as a standard tool for fMRI data analysis
\cite{worsley-etal:1996,ashburner:2011}.
The derivation in \cite{worsley:1994} was generalized in \cite{adler-taylor:2007} under the name \emph{Gaussian kinematic formula}.
An open problem is to determine whether the Euler characteristic heuristic is tail-valid for such random fields.

\bibliographystyle{abbrv}
\bibliography{spherical_tube-Rev-bib}

@article {adler-samorodnitsky-taylor:2010,
   AUTHOR = {Adler, Robert J. and Samorodnitsky, Gennady and Taylor, Jonathan E.},
    TITLE = {Excursion sets of three classes of stable random fields},
  JOURNAL = {Advances in Applied Probability},
 FJOURNAL = {Advances in Applied Probability},
   VOLUME = {42},
     YEAR = {2010},
   NUMBER = {2},
    PAGES = {293--318},
     ISSN = {0001-8678,1475-6064},
  MRCLASS = {60G52 (60D05 60G10 60G60)},
 MRNUMBER = {2675103},
MRREVIEWER = {Xiaowen\ Zhou},
}

@article {adler-samorodnitsky-taylor:2013,
   AUTHOR = {Adler, Robert J. and Samorodnitsky, Gennady and Taylor, Jonathan E.},
    TITLE = {High level excursion set geometry for non-{G}aussian infinitely divisible random fields},
  JOURNAL = {The Annals of Probability},
 FJOURNAL = {The Annals of Probability},
   VOLUME = {41},
     YEAR = {2013},
   NUMBER = {1},
    PAGES = {134--169},
     ISSN = {0091-1798,2168-894X},
  MRCLASS = {60G60 (60D05 60G10 60G17)},
 MRNUMBER = {3059195},
MRREVIEWER = {Alexander\ V.\ Bulinski\u i},
}

@book {adler-taylor:2007,
   author = "Adler, R. and Taylor, J.",
    title = "Random Fields and Geometry",
   SERIES = {Springer Monographs in Mathematics},
publisher = "Springer",
  address = "New York",
     year = "2007",
}

@article {ashburner:2011,
    title = {{SPM: A} history},
  journal = {NeuroImage},
   volume = {62},
   number = {2},
    pages = {791-800},
     year = {2012},
     issn = {1053-8119},
   author = {John Ashburner},
}

@book {bingham-etal:1987,
    AUTHOR = {Bingham, N. H. and Goldie, C. M. and Teugels, J. L.},
     TITLE = {Regular Variation},
    SERIES = {Encyclopedia of Mathematics and its Applications},
    VOLUME = {27},
 PUBLISHER = {Cambridge University Press, Cambridge},
      YEAR = {1987},
     PAGES = {xx+491},
      ISBN = {0-521-30787-2},
   MRCLASS = {26A12 (11K65 11N60 30-02 40E05 60-02 60Fxx)},
  MRNUMBER = {898871},
}

@article {cambanis-etal:1981,
    AUTHOR = {Cambanis, S. and Huang, S. and Simons, G.},
     TITLE = {On the theory of elliptically contoured distributions},
   JOURNAL = {Journal of Multivariate Analysis},
  FJOURNAL = {Journal of Multivariate Analysis},
    VOLUME = {11},
      YEAR = {1981},
    NUMBER = {3},
     PAGES = {368--385},
}

@book {foss-etal:2013,
    AUTHOR = {Foss, Sergey and Korshunov, Dmitry and Zachary, Stan},
     TITLE = {An Introduction to Heavy-Tailed and Subexponential Distributions},
    SERIES = {Springer Series in Operations Research and Financial Engineering},
   EDITION = {Second},
 PUBLISHER = {Springer, New York},
      YEAR = {2013},
     PAGES = {xii+157},
      ISBN = {978-1-4614-7100-4; 978-1-4614-7101-1},
   MRCLASS = {60E05 (60E10 60G50 60G70 62E15)},
  MRNUMBER = {3097424},
}

@incollection {goldie-kluppelberg:1998,
    AUTHOR = {Goldie, Charles M. and Kl\"{u}ppelberg, Claudia},
     TITLE = {Subexponential distributions},
 BOOKTITLE = {A practical guide to heavy tails ({S}anta {B}arbara, {CA}, 1995)},
     PAGES = {435--459},
 PUBLISHER = {Birkh\"{a}user Boston, Boston, MA},
      YEAR = {1998},
   MRCLASS = {60E05 (60G70 60J15 60J30 60K25 62P05)},
  MRNUMBER = {1652293},
}

@article{Gomez-etal:2008,
    author = {E. Gómez-Sánchez-Manzano and M. A. Gómez-Villegas and J. M. Marín},
     title = {Multivariate Exponential Power Distributions as Mixtures of Normal Distributions with Bayesian Applications},
   journal = {Communications in Statistics - Theory and Methods},
    volume = {37},
    number = {6},
     pages = {972--985},
      year = {2008},
 publisher = {Taylor \& Francis},
}

@article {hult-lindskog:2002,
    AUTHOR = {Hult, Henrik and Lindskog, Filip},
     TITLE = {Multivariate extremes, aggregation and dependence in elliptical distributions},
   JOURNAL = {Advances in Applied Probability},
  FJOURNAL = {Advances in Applied Probability},
    VOLUME = {34},
      YEAR = {2002},
    NUMBER = {3},
     PAGES = {587--608},
      ISSN = {0001-8678,1475-6064},
   MRCLASS = {60E05 (62H20)},
  MRNUMBER = {1929599},
MRREVIEWER = {Moshe\ Shaked},
}

@article {jensen-foutz:1989,
    author = {Jensen, D. R. and Foutz, R V.},
     title = {The Structure and Analysis of Spherical Time-Dependent Processes},
   journal = {SIAM Journal on Applied Mathematics},
    volume = {49},
    number = {6},
     pages = {1834-1844},
      year = {1989},
}

@article {johansen-johnstone:1990,
    AUTHOR = {Johansen, S. and Johnstone, I. M.},
     TITLE = {Hotelling's theorem on the volume of tubes: some illustrations in simultaneous inference and data analysis},
   JOURNAL = {The Annals of Statistics},
  FJOURNAL = {The Annals of Statistics},
    VOLUME = {18},
      YEAR = {1990},
    NUMBER = {2},
     PAGES = {652--684},
}

@article {johnstone-siegmund:1989,
    author = {Iain Johnstone and David Siegmund},
     title = {On {H}otelling's Formula for the Volume of Tubes and {N}aiman's Inequality},
    volume = {17},
   journal = {The Annals of Statistics},
    number = {1},
 publisher = {Institute of Mathematical Statistics},
     pages = {184--194},
      year = {1989},
}

@article {kuriki-takemura:2001,
    AUTHOR = {Kuriki, Satoshi and Takemura, Akimichi},
     TITLE = {Tail probabilities of the maxima of multilinear forms and their applications},
   JOURNAL = {The Annals of Statistics},
  FJOURNAL = {The Annals of Statistics},
    VOLUME = {29},
      YEAR = {2001},
    NUMBER = {2},
     PAGES = {328--371},
      ISSN = {0090-5364,2168-8966},
   MRCLASS = {62H10 (62H15)},
  MRNUMBER = {1863962},
}

@article {kuriki-takemura:2008,
    AUTHOR = {Kuriki, Satoshi and Takemura, Akimichi},
     TITLE = {Euler characteristic heuristic for approximating the distribution of the largest eigenvalue of an orthogonally invariant random matrix},
   JOURNAL = {Journal of Statistical Planning and Inference},
  FJOURNAL = {Journal of Statistical Planning and Inference},
    VOLUME = {138},
      YEAR = {2008},
    NUMBER = {11},
     PAGES = {3357--3378},
      ISSN = {0378-3758},
   MRCLASS = {60G60 (15A52 62M40)},
  MRNUMBER = {2450081},
MRREVIEWER = {Vasily A. Chernecky},
}

@article {kuriki-takemura:2009,
    AUTHOR = {Kuriki, Satoshi and Takemura, Akimichi},
     TITLE = {Volume of tubes and the distribution of the maximum of a {G}aussian random field},
   JOURNAL = {Selected Papers on Probability and Statistics, American Mathematical Society Translations Series 2},
    VOLUME = {227},
      YEAR = {2009},
    NUMBER = {2},
     PAGES = {25--48},
}

@book {liu:2010,
    AUTHOR = {Liu, W.},
     TITLE = {Simultaneous Inference in Regression},
 PUBLISHER = {Chapman \& Hall/CRC},
      YEAR = {2010},
}

@article {lu-kuriki:2017,
    AUTHOR = {Lu, X. and Kuriki, S.},
     TITLE = {Simultaneous confidence bands for contrasts between several nonlinear regression curves},
   JOURNAL = {Journal of Multivariate Analysis},
  FJOURNAL = {Journal of Multivariate Analysis},
    VOLUME = {155},
      YEAR = {2017},
     PAGES = {83--104},
}

@book {marinucci-peccati:2011,
    AUTHOR = {Marinucci, Domenico and Peccati, Giovanni},
     TITLE = {Random fields on the sphere: Representation, limit theorems and cosmological applications},
    SERIES = {London Mathematical Society Lecture Note Series},
    VOLUME = {389},
 PUBLISHER = {Cambridge University Press, Cambridge},
      YEAR = {2011},
     PAGES = {xii+341},
      ISBN = {978-0-521-17561-6},
   MRCLASS = {60G60 (60B15 60D05 60H05 62M15 85A40)},
  MRNUMBER = {2840154},
MRREVIEWER = {Anatoliy\ Malyarenko},
}

@article {naiman:1986,
    AUTHOR = {Naiman, D. Q.},
     TITLE = {Conservative confidence bands in curvilinear regression},
   JOURNAL = {The Annals of Statistics},
  FJOURNAL = {The Annals of Statistics},
    VOLUME = {14},
      YEAR = {1986},
    NUMBER = {3},
     PAGES = {896--906},
}

@article {savage:1962,
    AUTHOR = {Savage, I. R.},
     TITLE = {Mills' ratio for multivariate normal distributions},
   JOURNAL = {Journal of Research of the National Bureau of Standards},
  FJOURNAL = {Journal of Research of the National Bureau of Standards},
    VOLUME = {66},
      YEAR = {1962},
     PAGES = {93--96},
      ISSN = {},
   MRCLASS = {},
}

@article {schmidt:2002,
    AUTHOR = {Schmidt, Rafael},
     TITLE = {Tail dependence for elliptically contoured distributions},
      NOTE = {Special issue on mathematical models in market and credit risk},
   JOURNAL = {Mathematical Methods of Operations Research},
  FJOURNAL = {Mathematical Methods of Operations Research},
    VOLUME = {55},
      YEAR = {2002},
    NUMBER = {2},
     PAGES = {301--327},
      ISSN = {1432-2994},
   MRCLASS = {62E10 (60E05 62P05)},
  MRNUMBER = {1919580},
MRREVIEWER = {Barry C. Arnold},
}

@article{sun:1993,
      ISSN = {00911798, 2168894X},
    author = {Jiayang Sun},
   journal = {The Annals of Probability},
    number = {1},
     pages = {34--71},
 publisher = {Institute of Mathematical Statistics},
     title = {Tail Probabilities of the Maxima of {G}aussian Random Fields},
    volume = {21},
      year = {1993}
}

@article {sun-loader:1994,
    AUTHOR = {Sun, J.  and Loader, C. R.},
     TITLE = {Simultaneous confidence bands for linear regression and smoothing},
   JOURNAL = {The Annals of Statistics},
  FJOURNAL = {The Annals of Statistics},
    VOLUME = {22},
      YEAR = {1994},
    NUMBER = {3},
     PAGES = {1328--1345},
}

@article {takemura-kuriki:2003,
    AUTHOR = {Takemura, Akimichi and Kuriki, Satoshi},
     TITLE = {Tail probability via the tube formula when the critical radius is zero},
   JOURNAL = {Bernoulli},
  FJOURNAL = {Bernoulli. Official Journal of the Bernoulli Society for Mathematical Statistics and Probability},
    VOLUME = {9},
      YEAR = {2003},
    NUMBER = {3},
     PAGES = {535--558},
      ISSN = {1350-7265},
   MRCLASS = {60D05 (52A22 60G15 60G60)},
  MRNUMBER = {1997496},
MRREVIEWER = {N. Leonenko},
}

@article {takemura-kuriki:2002,
    AUTHOR = {Takemura, A. and Kuriki, S.},
     TITLE = {On the equivalence of the tube and {E}uler characteristic methods for the distribution of the maximum of {G}aussian fields over piecewise smooth domains},
   JOURNAL = {The Annals of Applied Probability},
  FJOURNAL = {The Annals of Applied Probability},
    VOLUME = {12},
      YEAR = {2002},
    NUMBER = {2},
     PAGES = {768--796},
}

@article {taylor-takemura-adler:2005,
    AUTHOR = {Taylor, Jonathan and Takemura, Akimichi and Adler, Robert J.},
     TITLE = {Validity of the expected {E}uler characteristic heuristic},
   JOURNAL = {The Annals of Probability},
  FJOURNAL = {The Annals of Probability},
    VOLUME = {33},
      YEAR = {2005},
    NUMBER = {4},
     PAGES = {1362--1396},
      ISSN = {0091-1798,2168-894X},
   MRCLASS = {60G15 (58J65 60G17 60G60)},
  MRNUMBER = {2150192},
MRREVIEWER = {Mikhail\ A.\ Lifshits},
}

@article{wells-etal:1962,
     ISSN = {00034851},
   author = {W. T. Wells and R. L. Anderson and John W. Cell},
  journal = {The Annals of Mathematical Statistics},
   number = {3},
    pages = {1016--1020},
publisher = {Institute of Mathematical Statistics},
    title = {The Distribution of the Product of Two Central or Non-Central Chi-Square Variates},
 urldate = {2025-02-20},
  volume = {33},
    year = {1962},
}

@article {worsley:1994,
    ISSN = {00018678},
  author = {K. J. Worsley},
 journal = {Advances in Applied Probability},
  number = {1},
   pages = {13--42},
publisher= {Applied Probability Trust},
   title = {Local Maxima and the Expected Euler Characteristic of Excursion Sets of $\chi^2$, $F$ and $t$ Fields},
 urldate = {2025-05-31},
  volume = {26},
    year = {1994},
}

@article {worsley-etal:1996,
  author = {Worsley, K. J. and Marrett, S. and Neelin, P. and Vandal, A. C. and Friston, K. J. and Evans, A. C.},
   title = {A unified statistical approach for determining significant signals in images of cerebral activation},
 journal = {Human Brain Mapping},
  volume = {4},
  number = {1},
   pages = {58--73},
    year = {1996},
}

@article {cistjakov:1964,
    AUTHOR = {\v{C}istjakov, V. P.},
     TITLE = {A theorem on sums of independent positive random variables and its applications to branching random processes},
   JOURNAL = {Akademija Nauk SSSR. Teorija Verojatnoste\u{\i} i ee Primenenija},
  FJOURNAL = {Akademija Nauk SSSR. Teorija Verojatnoste\u{\i} i ee Primenenija},
    VOLUME = {9},
      YEAR = {1964},
     PAGES = {710--718},
      ISSN = {0040-361x},
   MRCLASS = {60.67},
  MRNUMBER = {0170394},
MRREVIEWER = {K. Dietz},
}

\end{document}